\documentclass[11pt]{amsart}
\usepackage[T1]{fontenc}
\usepackage[utf8]{inputenc}
\usepackage{lmodern}
\usepackage{microtype}
\usepackage{amsmath,amssymb,amsthm,mathtools}
\usepackage{enumitem}
\usepackage{textcase}
\usepackage[colorlinks=false,pdfborder={0 0 1},linkbordercolor={1 0 0},citebordercolor={1 0 0},urlbordercolor={1 0 0},filebordercolor={1 0 0},breaklinks=true,linktoc=page]{hyperref}
\usepackage{aliascnt}
\usepackage{geometry}
\numberwithin{equation}{section}

\newtheorem{theorem}{Theorem}[section]
\newaliascnt{proposition}{theorem}
\newtheorem{proposition}[proposition]{Proposition}
\aliascntresetthe{proposition}
\newaliascnt{lemma}{theorem}
\newtheorem{lemma}[lemma]{Lemma}
\aliascntresetthe{lemma}
\newaliascnt{corollary}{theorem}
\newtheorem{corollary}[corollary]{Corollary}
\aliascntresetthe{corollary}
\newaliascnt{definition}{theorem}
\newtheorem{definition}[definition]{Definition}
\aliascntresetthe{definition}
\newaliascnt{remark}{theorem}
\newtheorem{remark}[remark]{Remark}
\aliascntresetthe{remark}

\DeclareMathOperator{\ordop}{ord}
\newcommand{\gr}{\operatorname{gr}}
\newcommand{\Qp}{\mathbb{Q}_p}
\newcommand{\Zp}{\mathbb{Z}_p}
\newcommand{\Fp}{\mathbb{F}_p}

\newcommand{\Sp}{S_p}
\newcommand{\Tp}{T_p}

\newcommand{\bigo}{O}
\newcommand{\vp}{\ordop_p}
\newcommand{\wt}{\operatorname{wt}}

\title{Coefficient-Level B\"ottcher Theory for Wild Superattracting Germs of Degree \NoCaseChange{\texorpdfstring{$p^e$}{pe}}}
\author{Rufei Ren}
\address{Department of Mathematics, Fudan University}
\email{rufeir@fudan.edu.cn}

\begin{document}
\begin{abstract}
Let $p$ be an odd prime, let $e\ge2$, and put $q=p^e$.  We study the wild family
\[
\varphi_{r,e}(x)=x^q+qp^r x^{q+1}=x^{p^e}+p^{r+e}x^{p^e+1}
\qquad (r\ge0),
\]
and the inverse B\"ottcher coordinate $f_{r,e}(x)=x\sum_{k\ge0}a_k(r,e)x^k/k!$ characterized by
\[
\varphi_{r,e}(f_{r,e}(x))=f_{r,e}(x^q).
\]
For the clean family, we prove a complete mod-$p$ digit-sum law in the special fiber $r=0$.  For the higher fibers $r\ge1$, we prove a coefficient-level theorem consisting of a global digit-weight lower bound, a leading monomial theorem on divisible non-pure classes, a lag-$e$ pure-power recursion, and subadditivity of the induced digit weight.  This yields the pure-power branch word
\[
(B^{e-1}A)^{\lceil r/e\rceil}B^\infty
\]
and the radius formula
\[
\rho(f_{r,e})=p^{-\theta_{r,e}},\qquad \theta_{r,e}=p^{-e\lceil r/e\rceil}\left(\frac{1}{p-1}+e\lceil r/e\rceil-r\right).
\]
We then prove a tail-stable extension.  In the special fiber, $p$-divisible tails preserve the digit-sum law modulo $p$.  In the higher fibers, tails satisfying $\vp(\vartheta_h)\ge\Lambda_{r,e}(h+1)+1$ lie beyond the clean-family initial $\Lambda_{r,e}$-graded term and therefore preserve the leading terms, the pure-power branch word, the valuation asymptotic, and the radius.  For $e=2$, this recovers the Salerno--Silverman degree-$p^2$ family and the Fu--Nie radius statement for the inverse coordinate in that family.
\end{abstract}

\maketitle
\tableofcontents

\section{Introduction}
The local study of superattracting germs begins with B\"ottcher's theorem \cite{Bottcher1904}, which gives a canonical coordinate in which a germ is conjugate to a pure monomial.  In one complex variable this coordinate is central both near a superattracting fixed point and, for polynomials, near infinity, where it encodes the escape-rate function and the external geometry of Julia sets; see, for example, \cite{AlexanderIavernaroRosa,DouadyHubbard1,DouadyHubbard2,Milnor,Ritt}.  The same normal form also appears in higher-dimensional holomorphic dynamics \cite{BuffEpsteinKoch, HubbardPapadopol}.

Over non-archimedean fields, B\"ottcher coordinates are part of a broader local and global theory of rational dynamics over valued fields; see, for example, \cite{BakerRumely,Benedetto,FavreRiveraLetelier, Lubin,RiveraLetelier,SilvermanADS}.  In this setting the coefficients of the coordinate carry arithmetic information, reflecting ramification, integrality, and valuation phenomena.  In polynomial dynamics, $p$-adic B\"ottcher coordinates have also been used in the study of arboreal Galois representations and bounded-height problems in families \cite{DeMarcoGhiocaKriegerNguyenTuckerYe,Ingram}.  Fix an odd prime $p$, an integer $e\ge 2$, and put $q=p^e$.  In this paper we study the wild one-parameter family
\begin{equation}\label{eq:main-family}
        \varphi_{r,e}(x)=x^q+qp^r x^{q+1}=x^{p^e}+p^{r+e}x^{p^e+1},
        \qquad q=p^e,
\end{equation}
where $r\ge 0$, together with its inverse B\"ottcher coordinate
\[
        f_{r,e}(x)=x\sum_{k\ge 0}\frac{a_k(r,e)}{k!}x^k,
        \qquad a_0(r,e)=1,
\]
normalized by
\[
        \varphi_{r,e}\bigl(f_{r,e}(x)\bigr)=f_{r,e}(x^q).
\]
When $e$ is fixed we suppress it from the notation and write $a_k(r)$ and $f_r$.  Equivalently, if $\Phi_{r,e}=f_{r,e}^{-1}$, then
\[
        \Phi_{r,e}\bigl(\varphi_{r,e}(x)\bigr)=\Phi_{r,e}(x)^q.
\]

Our aim is to describe the coefficients $a_k(r,e)$ directly in degree $p^e$.  In the special fiber $r=0$ we determine all coefficients modulo $p$ by a closed digit-sum law.  In the higher fibers $r\ge1$ we prove a coefficient-level theorem consisting of a global digit-weight lower bound, a leading monomial theorem on divisible non-pure classes, a pure-power recursion with lag $e$, and an exact radius formula.  We formulate the degree-$p^e$ problem first and view the degree-$p^2$ case only as the specialization $e=2$.

The case $e=2$ is the family isolated by Salerno and Silverman in their wild conjectures \cite{SalernoSilverman}.  Fu and Nie proved the radius in a substantially broader wild superattracting setting \cite{FuNie}.  In their notation the B\"ottcher coordinate $\Phi$ is our $f_{r,2}^{-1}$, so the radius of $\Phi^{-1}$ in \cite{FuNie} is exactly the radius of our $f_{r,2}$.  The formulas proved below for $e>2$ belong to the present extension.

\medskip
\noindent\emph{Why the degree-$p^e$ theorem is not formal.}
The higher-fiber theorem is driven by the pure-power recursion.  For general $e$ this recursion has lag $e$, and its branch word is
\[
        (B^{e-1}A)^sB^\infty,
        \qquad s=\left\lceil\frac r e\right\rceil.
\]
This branch word is what produces both the stable pure-power slope and the exponent in the radius formula.  Only after the general pattern has been established does one recover the degree-$p^2$ case by setting $e=2$.

\medskip
\noindent\emph{Method.}
Our higher-fiber argument is built around a filtered cumulant principle.  We write the $B$-coefficient valuation as a sum of two carry defects.  In the divisible non-pure sector we then show that every unit scalar term is already carry-free, and from this point the degree-$\Lambda_{r,e}$ initial unit sector is read off from the carry-free cumulant coefficient.  This is the key step behind the leading monomial theorem on divisible classes.  At stable layers one may still use the conceptual quotient with relations $Y_j^p=Y_{j+1}$, but the scalar calculation is finished before passing to that quotient.

The clean-family results are the main theorems of the paper.  After stating them, we turn to a perturbative extension showing that the same coefficient-level structure remains stable under sufficiently small higher-order tails.  In the higher-fiber statement, the condition
\[
        \vp(\vartheta_h)\ge \Lambda_{r,e}(h+1)+1
\]
is chosen so that every tail contribution lies strictly above the initial $\Lambda_{r,e}$-graded term.

We first state the special-fiber congruence law.  In this statement $a_k=a_k(0,e)$.

\begin{theorem}[Digit-sum formula in the special fiber]\label{thm:digit-sum}
Let $k\ge 1$ and write
\[
        k=a+d_1p+d_2p^2+\cdots+d_Np^N,
        \qquad 0\le a,d_i\le p-1,
\]
with
\[
        s:=d_1+\cdots+d_N.
\]
Then
\begin{equation}\label{eq:digit-formula-aa}
        a_k\equiv (-1)^s(a+1)^s a_a \pmod p.
\end{equation}
In fact,
\begin{equation}\label{eq:closed-digit-formula}
        a_k\equiv (-1)^{a+s}(a+1)^{a+s-1}\pmod p.
\end{equation}
\end{theorem}

As a first corollary, one obtains explicit congruences in the three residue classes $0,-1,-2\pmod p$.

\begin{theorem}[Special-fiber residue classes]\label{thm:SS-special}
For every integer $m\ge 1$ one has
\[
        a_{pm}\equiv (-1)^m,
        \qquad
        a_{pm-1}\equiv 0,
        \qquad
        a_{pm-2}\equiv -1\pmod p.
\]
\end{theorem}

We then state the recursive theorem for the higher fibers.  Fix $r\ge 1$.  Put
\[
        v_i(r,e):=\vp\bigl(a_{p^i}(r,e)\bigr)\qquad (i\ge 0),
\]
and, when $e$ is fixed, write simply $v_i(r)$.  For $k=\sum_{i\ge 0}k_ip^i$, define
\[
        \Lambda_{r,e}(k):=\sum_{i\ge 0}k_iv_i(r,e),
        \qquad
        m_{r,e}(k):=\prod_{i\ge 0}a_{p^i}(r,e)^{k_i}.
\]
Set
\[
        \mu_e:=\frac{p^e-1}{p-1}=1+p+\cdots+p^{e-1}.
\]
For $n\ge e$ define
\[
\begin{aligned}
        \Delta_{n,e}&:=\mu_e p^{n-e}-e,\\
        A_{n,e}(r)&:=\Delta_{n,e}+v_{n-e}(r,e),\\
        B_{n,e}(r)&:=p\,v_{n-1}(r,e),
\end{aligned}
\]
and for $n<e$ regard $A_{n,e}$ as absent.  Finally set
\[
\begin{aligned}
        \alpha_{n,e}&:=\frac{(p^n)!}{p^e(p^{n-e})!} && (n\ge e),\\
        \gamma_{n,e}&:=\frac{(p^n)!}{p\bigl((p^{n-1})!\bigr)^p}\binom{p^e-1}{p-1} && (n\ge 1).
\end{aligned}
\]

\begin{theorem}[Recursive structure in degree $p^e$]\label{thm:higher-structure}
For every fixed $r\ge 1$ and $e\ge 2$, the following assertions hold.
\begin{enumerate}[label=\textup{(\arabic*)},leftmargin=2.2em]
\item \emph{Global digit-weight lower bound.}  For every $k\ge 1$,
\[
        \vp\bigl(a_k(r,e)\bigr)\ge \Lambda_{r,e}(k).
\]
\item \emph{Leading monomial on divisible non-pure classes.}  If $p\mid k$ and $k$ is not a power of $p$, then
\[
        a_k(r,e)\equiv m_{r,e}(k)\pmod {p^{\Lambda_{r,e}(k)+1}}.
\]
\item \emph{Pure-power recursion.}  One has $v_0(r,e)=r$ and, for $1\le n<e$,
\[
        v_n(r,e)=p\,v_{n-1}(r,e)=p^nr.
\]
For every $n\ge e$,
\[
        v_n(r,e)=\min\{A_{n,e}(r),B_{n,e}(r)\},
        \qquad
        A_{n,e}(r)\ne B_{n,e}(r).
\]
Moreover the corresponding leading branch is
\[
        a_{p^n}(r,e)=
        \begin{cases}
        -\gamma_{n,e}a_{p^{n-1}}(r,e)^p+\bigo\bigl(p^{v_n(r,e)+1}\bigr),&1\le n<e,\\[1mm]
        \alpha_{n,e}a_{p^{n-e}}(r,e)+\bigo\bigl(p^{v_n(r,e)+1}\bigr),& n\ge e,\ A_{n,e}(r)<B_{n,e}(r),\\[1mm]
        -\gamma_{n,e}a_{p^{n-1}}(r,e)^p+\bigo\bigl(p^{v_n(r,e)+1}\bigr),& n\ge e,\ B_{n,e}(r)<A_{n,e}(r).
        \end{cases}
\]
In particular $v_n(r,e)\le p\,v_{n-1}(r,e)$ for all $n\ge 1$.
\item \emph{Subadditivity of the digit weight.}  For every finite sum $N=n_1+\cdots+n_t$ of nonnegative integers,
\[
        \Lambda_{r,e}(N)\le \Lambda_{r,e}(n_1)+\cdots+\Lambda_{r,e}(n_t).
\]
\end{enumerate}
\end{theorem}

The following theorem gives the explicit branch pattern and the radius.  Put
\[
        s_{r,e}:=\left\lceil\frac r e\right\rceil,
        \qquad
        \lambda_{r,e}:=p^{-e s_{r,e}}
        \left(r+\frac{p^{e s_{r,e}}-1}{p-1}-e\,s_{r,e}\right),
\]
and
\[
        \theta_{r,e}:=\frac1{p-1}-\lambda_{r,e}
        =p^{-e s_{r,e}}\left(\frac1{p-1}+e\,s_{r,e}-r\right).
\]
In particular, $\theta_{r,e}>0$, since $e\,s_{r,e}-r\ge0$.
For $n\ge 0$ define
\[
        N_{r,e}(n):=\min\left(\left\lfloor\frac n e\right\rfloor,s_{r,e}\right),
        \qquad
        \varepsilon_{r,e}(n):=(-1)^{1+eN_{r,e}(n)}\in\{\pm1\}.
\]
Here $N_{r,e}(n)$ is the number of $A$-branches encountered up to level $n$.

\begin{theorem}[Valuations and radius in degree $p^e$]\label{thm:radius-pe}
Fix $r\ge 1$ and $e\ge 2$.
\begin{enumerate}[label=\textup{(\alph*)},leftmargin=2.2em]
\item The pure-power branch word is
\[
        (B^{e-1}A)^{s_{r,e}}B^\infty.
\]
Equivalently, for $0\le j\le s_{r,e}$ and $0\le t\le e-1$,
\[
        v_{je+t}(r,e)=p^t\left(r+\frac{p^{j e}-1}{p-1}-e\,j\right),
\]
and for all $n\ge es_{r,e}$,
\[
        v_n(r,e)=\lambda_{r,e}p^n.
\]
\item For every $n\ge 0$,
\[
        p^{-v_n(r,e)}a_{p^n}(r,e)\equiv \varepsilon_{r,e}(n)\pmod p.
\]
Consequently, if $m=\sum_{i\ge 0}m_ip^i$, then
\[
        a_{pm}(r,e)=
        \left(\prod_{i\ge 0}\varepsilon_{r,e}(i+1)^{m_i}\right)
        p^{\sum_{i\ge 0}m_i v_{i+1}(r,e)}
        +\bigo\!\left(p^{\sum_{i\ge 0}m_i v_{i+1}(r,e)+1}\right).
\]
\item If $p\mid k$, then
\[
        \vp\bigl(a_k(r,e)\bigr)=\Lambda_{r,e}(k)=\lambda_{r,e}k+\bigo(1).
\]
\item The $p$-adic radius of convergence of $f_{r,e}$ is
\[
        \rho(f_{r,e})=p^{-\theta_{r,e}}
        =p^{-\left(p^{-e s_{r,e}}\left(\frac1{p-1}+e\,s_{r,e}-r\right)\right)}.
\]
For $e=2$, this is $p^{-\frac{p^{-r}}{p-1}}$.
\end{enumerate}
\end{theorem}

\begin{theorem}[Tail-stable extension]\label{thm:tail-main}
Let $\vartheta_h\in\Qp$ for all $h\ge1$, and consider the formal germ
\[
        \widetilde\varphi_{r,e}(x)=x^q+qp^r x^{q+1}+q\sum_{h\ge1}\vartheta_hx^{q+1+h},
        \qquad q=p^e,
\]
with inverse B\"ottcher coordinate
\[
        \widetilde f_{r,e}(x)=x\sum_{k\ge0}\frac{\widetilde a_k(r,e)}{k!}x^k,
        \qquad
        \widetilde\varphi_{r,e}\bigl(\widetilde f_{r,e}(x)\bigr)=\widetilde f_{r,e}(x^q).
\]
Since the coefficient recursion at a fixed degree involves only finitely many $h$, all statements below are understood coefficient-wise.
\begin{enumerate}[label=\textup{(\alph*)},leftmargin=2.2em]
\item If $r=0$ and $\vartheta_h\in p\Zp$ for all $h\ge1$, then
\[
        \widetilde a_k(0,e)\equiv a_k(0,e)\pmod p
        \qquad (k\ge0).
\]
In particular, the conclusions of Theorems~\ref{thm:digit-sum} and~\ref{thm:SS-special} remain valid for $\widetilde a_k(0,e)$.
\item Assume $r\ge1$ and
\[
        \vp(\vartheta_h)\ge \Lambda_{r,e}(h+1)+1
        \qquad (h\ge1).
\]
For $k=\sum_{i\ge0}k_ip^i$, set
\[
        \widetilde m_{r,e}(k):=\prod_{i\ge0}\widetilde a_{p^i}(r,e)^{k_i}.
\]
Here $\Lambda_{r,e}$ is the clean-family digit weight from Theorem~\ref{thm:higher-structure}.  Then the analogues of Theorems~\ref{thm:higher-structure} and~\ref{thm:radius-pe} hold for the perturbed coefficients $\widetilde a_k(r,e)$, with the same digit weight $\Lambda_{r,e}$ and with $\widetilde m_{r,e}(k)$ in place of $m_{r,e}(k)$.  In particular, the pure-power branch word, the normalized pure-power units, the valuation asymptotic, and the radius $p^{-\theta_{r,e}}$ are unchanged.
\end{enumerate}
\end{theorem}

\begin{remark}[The degree-$p^2$ specialization]\label{rem:e2-ss-conjectures}
When $e=2$, the family \eqref{eq:main-family} becomes
\[
        \varphi_{r,2}(x)=x^{p^2}+p^{r+2}x^{p^2+1},
\]
which is exactly the wild degree-$p^2$ family studied by Salerno--Silverman.  In this specialization, Theorem~\ref{thm:SS-special} proves the special-fiber residue-class prediction of Salerno--Silverman, while Theorem~\ref{thm:radius-pe} proves the higher-fiber valuation and radius prediction in the same family.  The stronger results Theorems~\ref{thm:digit-sum} and~\ref{thm:higher-structure} may be viewed as coefficient-level refinements of those two predictions.
\end{remark}

Although Theorem~\ref{thm:tail-main} is far from the full generality of \cite{FuNie}, it shows that the coefficient-level structure developed here is not confined to the exact one-parameter model \eqref{eq:main-family}: the special-fiber digit-sum law and the higher-fiber leading-term calculus persist under controlled higher-order perturbations.

Throughout the paper, if $X,Y\in\Qp$ and $N\in\mathbb Z$, we write
\[
        X=Y+\bigo(p^N)
\]
to mean $\vp(X-Y)\ge N$.  All congruences modulo powers of $p$ are applied only after the relevant quantities have been shown to be $p$-integral.

The paper is arranged as follows.  Sections~\ref{sec:recursion}--\ref{sec:proof-main-special} deal with the special fiber.  We prove the $A$--$B$--$C$ recursion, the carry-defect integrality of the $B$-term, the residue class $a=0$, and the vector-partition cumulant collapse for $a\ne0$.  Sections~\ref{sec:higher-recursion}--\ref{sec:higher-induction} deal with the higher fibers.  We first set up the generalized recursion and the filtered cumulant argument, and then we prove Theorem~\ref{thm:higher-structure}, the branch word, and Theorem~\ref{thm:radius-pe}.  Section~\ref{sec:tail-stability} proves the tail-stable extension stated in Theorem~\ref{thm:tail-main}.

We restrict throughout to odd primes.  The case $p=2$ brings in additional low-characteristic coincidences, and for this reason we do not discuss it here.

\section{The recursion and structural lemmas}\label{sec:recursion}
In this section we work in the special fiber $r=0$.  Thus
\[
        \varphi(x)=x^q+qx^{q+1},\qquad q=p^e,
\]
and we write $a_k=a_k(0,e)$.  Since $\varphi(x)$ is of the form $x^m+mx^{m+1}R[\![x]\!]$ with $m=q$ and $R=\Zp$, the coefficient-integrality theorem of Salerno--Silverman, namely \cite[Theorem~3(b)]{SalernoSilverman}, gives $a_k\in\Zp$ for all $k$.  Therefore all reductions modulo $p$ in Sections~\ref{sec:recursion}--\ref{sec:proof-main-special} are legitimate.

The coefficient comparison gives
\begin{equation}\label{eq:ABC}
        a_k=A_k[x^k]-B_k[x^k]-C_k[x^k],
\end{equation}
where
\[
        A_k:=\frac{k!}{q}\sum_{\ell=0}^{k-1}\frac{a_\ell}{\ell!}x^{q\ell},\qquad
        B_k:=\frac{k!}{q}\left(\sum_{\ell=0}^{k-1}\frac{a_\ell}{\ell!}x^\ell\right)^q,
\]
\[
        C_k:=k!x\left(\sum_{\ell=0}^{k-1}\frac{a_\ell}{\ell!}x^\ell\right)^{q+1}.
\]
Here $P[x^n]$ denotes the coefficient of $x^n$ in $P(x)$.

For $N=\sum_iN_ip^i$, $0\le N_i\le p-1$, put
\[
        \Sp(N):=\sum_iN_i,
        \qquad
        \Tp(N):=\frac{N!}{p^{\vp(N!)}},
\]
so that $\Tp(N)$ is the prime-to-$p$ part of $N!$.  We repeatedly use Legendre's formula
\begin{equation}\label{eq:legendre}
        \vp(N!)=\frac{N-\Sp(N)}{p-1}
\end{equation}
and digit-sum subadditivity
\begin{equation}\label{eq:digit-subadditivity}
        \Sp(u+v)\le \Sp(u)+\Sp(v).
\end{equation}

\begin{proposition}[The $A$-term vanishes modulo $p$]\label{prop:A-zero}
For every $k\ge 1$,
\[
        A_k[x^k]\equiv 0\pmod p.
\]
\end{proposition}
\begin{proof}
If $q\nmid k$, then $A_k[x^k]=0$.  If $k=qn$, then
\[
        A_k[x^k]=\frac{(qn)!}{q\,n!}a_n.
\]
Since $\Sp(qn)=\Sp(n)$, Legendre's formula gives
\[
\begin{aligned}
\vp\left(\frac{(qn)!}{q\,n!}\right)
&=\frac{qn-\Sp(qn)}{p-1}-e-\frac{n-\Sp(n)}{p-1}  \\
&=\frac{(q-1)n}{p-1}-e
=\mu_en-e.
\end{aligned}
\]
Here $\mu_e=(q-1)/(p-1)\ge e+1$ for $p\ge3$ and $e\ge2$, so the last number is at least $1$.  Hence the $A$-term is divisible by $p$.
\end{proof}

Let $a^\eta=\prod_{n\ge0}a_n^{\eta_n}$ be a monomial occurring in $B_k[x^k]$.  Then
\[
        \sum_n\eta_n=q,
        \qquad
        \sum_n n\eta_n=k,
\]
and its scalar coefficient is
\[
        \gamma_B(\eta)=\frac{k!}{q}\binom{q}{\eta_0,\eta_1,\ldots}\prod_{n\ge0}\frac1{(n!)^{\eta_n}}.
\]
Legendre's formula gives
\begin{equation}\label{eq:B-valuation-general}
(p-1)\vp(\gamma_B(\eta))
=-\Sp(k)-1-e(p-1)+\sum_n\Sp(\eta_n)+\sum_n\eta_n\Sp(n).
\end{equation}
Equivalently, with
\[
        c(\eta):=\frac{\sum_n\eta_n\Sp(n)-\Sp(k)}{p-1},
        \qquad
        d(\eta):=\frac{\sum_n\Sp(\eta_n)-1}{p-1},
\]
one has
\begin{equation}\label{eq:B-cd}
        \vp(\gamma_B(\eta))=c(\eta)+d(\eta)-e.
\end{equation}
The integers $c(\eta)$ and $d(\eta)$ are the carry defects in the weighted addition $\sum n\eta_n=k$ and in the multiplicity addition $\sum\eta_n=q$.

\begin{lemma}[Carry depth for multiplicities]\label{lem:carry-depth}
Let $\eta_n\ge0$ and $\sum_n\eta_n=p^e$.  If $d(\eta)=t<e$, then every $\eta_n$ is divisible by $p^{e-t}$.  If $t\ge e$, the assertion is vacuous.
\end{lemma}
\begin{proof}
If $t\ge e$, there is nothing to prove.  Thus assume $t<e$.
Let
\[
        s_0:=\min_{\eta_n>0}\vp(\eta_n),
\]
and write $\eta_n=p^{s_0}u_n$ with $u_n\in\mathbb{Z}_{\ge0}$.  Then at least one $u_n$ is not divisible by $p$, and $\sum_nu_n=p^{e-s_0}$.  The addition of the $u_n$ to obtain $p^{e-s_0}$ has at least one nonzero units digit among the summands; hence it must carry from the units place.  Since the final number has all lower $e-s_0$ digits equal to zero, carries must propagate through levels $0,1,\ldots,e-s_0-1$.  Thus the number of carries is at least $e-s_0$.  But this number of carries is
\[
        \frac{\sum_n\Sp(u_n)-1}{p-1}
        =\frac{\sum_n\Sp(\eta_n)-1}{p-1}=d(\eta)=t.
\]
Therefore $s_0\ge e-t$.
\end{proof}

\begin{lemma}[Global $B$-coefficient integrality]\label{lem:B-integrality}
For every $k\ge1$ and every monomial coefficient $\gamma_B(\eta)$ occurring in $B_k[x^k]$,
\[
        c(\eta)+d(\eta)\ge e.
\]
Equivalently, $\gamma_B(\eta)\in\Zp$.
\end{lemma}
\begin{proof}
Put $t=d(\eta)$ and $s=e-t$.  If $s\le0$, there is nothing to prove.  By Lemma~\ref{lem:carry-depth}, $\eta_n=p^s u_n$ for all $n$.  Let
\[
        M:=\sum_nnu_n,
        \qquad
        A:=\sum_nu_n\Sp(n).
\]
Then $k=p^sM$ and $\Sp(k)=\Sp(M)$.  Digit-sum subadditivity gives $\Sp(M)\le A$, and since $k\ge1$ we have $A\ge1$.  Hence
\[
\begin{aligned}
        c(\eta)
        &=\frac{p^sA-\Sp(M)}{p-1}                         \\
        &\ge \frac{(p^s-1)A}{p-1}
        \ge \frac{p^s-1}{p-1}
        \ge s=e-t.
\end{aligned}
\]
Thus $c(\eta)+d(\eta)\ge e$.
\end{proof}

\begin{proposition}[Unit $B$-terms on divisible classes]\label{prop:B-unit-divisible}
Assume $p\mid k$ and let a monomial in $B_k[x^k]$ have $p$-adic unit scalar coefficient.  If some occurring positive index is not divisible by $p$, then necessarily $k=p$ and the monomial is
\[
        a_0^{q-p}a_1^p.
\]
Consequently, if $p\mid k$ and $k>p$, every unit-coefficient monomial in $B_k[x^k]$ uses only positive indices divisible by $p$.
\end{proposition}
\begin{proof}
The unit condition is $c(\eta)+d(\eta)=e$.  Suppose that some occurring positive index is not divisible by $p$.  If $d(\eta)=e$, then $c(\eta)=0$, so the addition $\sum n\eta_n=k$ is carry-free.  Since $p\mid k$, its units digit is zero, and carry-freeness forces every occurring positive index to have units digit zero, a contradiction.  Hence $d(\eta)<e$.

Put $s=e-d(\eta)\ge1$.  By Lemma~\ref{lem:carry-depth}, $\eta_n=p^s u_n$.  Let $M=\sum_nnu_n$ and $A=\sum_nu_n\Sp(n)$.  Since $c(\eta)=s$, we have
\[
        p^sA-\Sp(M)=s(p-1).
\]
As before, $\Sp(M)\le A$, and hence
\[
        (p^s-1)A\le s(p-1).
\]
For $s\ge2$ this is impossible, because $p^s-1>s(p-1)$ for odd $p$.  Therefore $s=1$, and then the inequality forces $A=1$ and $\Sp(M)=1$.  Thus exactly one positive index occurs after division by $p$, with multiplicity one and digit sum one.  Since we assumed an occurring positive index not divisible by $p$, this index must be $1$.  Hence $k=p$ and the corresponding multiplicities are $\eta_1=p$, $\eta_0=q-p$, as claimed.
\end{proof}

\begin{proposition}[A global $C$-term lemma]\label{prop:C-global}
For every $k\ge1$, every monomial coefficient of $C_k[x^k]$ is $p$-integral.  Moreover:
\begin{enumerate}[label=\textup{(\roman*)},leftmargin=2.2em]
\item if $p\mid k$, then every monomial coefficient is divisible by $p$, so $C_k[x^k]\equiv0\pmod p$;
\item if $k\equiv a\pmod p$ with $1\le a\le p-1$, then the unique surviving term modulo $p$ is $a\,a_{k-1}$, i.e.
\[
        C_k[x^k]\equiv a\,a_{k-1}\pmod p.
\]
\end{enumerate}
\end{proposition}
\begin{proof}
Let $a^\eta=\prod_{n=0}^{k-1}a_n^{\eta_n}$ occur in $C_k[x^k]$.  Then
\[
        \sum_{n=0}^{k-1}\eta_n=q+1,
        \qquad
        \sum_{n=0}^{k-1}n\eta_n=k-1,
\]
and
\[
        \gamma_C(\eta)=k!\binom{q+1}{\eta_0,\ldots,\eta_{k-1}}\prod_{n=0}^{k-1}\frac1{(n!)^{\eta_n}}.
\]
A Legendre calculation gives
\begin{equation}\label{eq:C-valuation}
(p-1)\vp(\gamma_C(\eta))=-\Sp(k)-1+\Sigma(\eta)+T(\eta),
\end{equation}
where
\[
        \Sigma(\eta):=\sum_{n=0}^{k-1}\Sp(\eta_n),
        \qquad
        T(\eta):=\sum_{n=0}^{k-1}\eta_n\Sp(n).
\]
Since $\Sp(q+1)=2$, digit-sum subadditivity gives $\Sigma(\eta)\ge2$ and $T(\eta)\ge\Sp(k-1)$.

If $p\mid k$, write $k=p^\nu u$ with $\nu\ge1$ and $p\nmid u$.  Then $\Sp(k-1)=\Sp(k)-1+\nu(p-1)$, so \eqref{eq:C-valuation} gives
\[
        (p-1)\vp(\gamma_C(\eta))\ge\nu(p-1),
\]
and every $C$-coefficient is divisible by $p$.

Assume now that $k\equiv a\not\equiv0\pmod p$.  Then $\Sp(k-1)=\Sp(k)-1$, so all coefficients are $p$-integral, and a coefficient survives modulo $p$ only if
\[
        \Sigma(\eta)=2,
        \qquad
        T(\eta)=\Sp(k-1).
\]
The condition $\Sigma(\eta)=2$ leaves only the patterns $\eta_t=q+1$ or $\eta_r=q$, $\eta_s=1$.  In the first case, if $t>0$, then
\[
        \Sp(k-1)=\Sp((q+1)t)\le\Sp(qt)+\Sp(t)=2\Sp(t)<(q+1)\Sp(t)=T(\eta),
\]
a contradiction.  Hence $t=0$, which only gives the immediate case $k=1$.  In the second case, if $r>0$, then
\[
        \Sp(k-1)=\Sp(qr+s)\le\Sp(r)+\Sp(s)<q\Sp(r)+\Sp(s)=T(\eta),
\]
again a contradiction.  Thus $r=0$, $s=k-1$, and the unique surviving monomial is $\eta_0=q$, $\eta_{k-1}=1$.  Its coefficient is
\[
        k!\binom{q+1}{q,1}\frac1{(k-1)!}=k(q+1)\equiv k\equiv a\pmod p.
\]
\end{proof}

\begin{theorem}[The case $a=0$]\label{thm:a-zero}
For every integer $M\ge1$,
\[
        a_{pM}\equiv (-1)^M\pmod p.
\]
\end{theorem}
We give the proof in the next section.

\section{The residue class \texorpdfstring{$a=0$}{a=0}}\label{sec:a-zero}
We prove Theorem~\ref{thm:a-zero}.  The same truncated exponential identity works for every $e\ge2$.

\begin{lemma}[A truncated-exponential coefficient identity]\label{lem:trunc-exp}
For $M\ge1$, set
\[
        S_{M-1}(y):=\sum_{i=0}^{M-1}\frac{(-1)^i}{p^ii!}y^i.
\]
Then
\[
        [y^M]S_{M-1}(y)^q
        =\frac{(-q/p)^M}{M!}-q\frac{(-1)^M}{p^M M!}.
\]
Equivalently,
\[
        [y^M]\left(\sum_{i=0}^{M-1}\frac{y^i}{p^ii!}\right)^q
        =\frac{p^{(e-1)M}}{M!}-\frac{q}{p^M M!}.
\]
\end{lemma}
\begin{proof}
Write $E(y)=e^{-y/p}$ and
\[
        R_M(y):=\sum_{i\ge M}\frac{(-1)^i}{p^ii!}y^i.
\]
Then $S_{M-1}=E-R_M$ and
\[
        -R_M(y)=-\frac{(-1)^M}{p^MM!}y^M+\bigo(y^{M+1}).
\]
In the expansion $(E-R_M)^q$, all terms with at least two copies of $R_M$ have $y$-adic order $>M$.  Hence
\[
\begin{aligned}
        [y^M]S_{M-1}(y)^q
        &=[y^M]E(y)^q-q\frac{(-1)^M}{p^MM!}        \\
        &=\frac{(-q/p)^M}{M!}-q\frac{(-1)^M}{p^MM!}.
\end{aligned}
\]
Replacing $y$ by $-y$ gives the second formula.
\end{proof}

\begin{lemma}\label{lem:first-a1-ap}
One has
\[
        a_1=-1,
        \qquad
        a_p\equiv -1\pmod p.
\]
\end{lemma}
\begin{proof}
Write $f(x)=x+a_1x^2+\bigo(x^3)$.  Comparing the coefficient of $x^{q+1}$ in
\[
        f(x)^q+qf(x)^{q+1}=f(x^q)
\]
gives $qa_1+q=0$, hence $a_1=-1$.

For $a_p$, Propositions~\ref{prop:A-zero} and \ref{prop:C-global} give $a_p\equiv -B_p[x^p]\pmod p$.  By Proposition~\ref{prop:B-unit-divisible}, the only unit-coefficient monomial is $a_0^{q-p}a_1^p$.  Its scalar is
\[
        \gamma_p=\frac{p!}{q}\binom{q}{q-p,p}=\frac{(q-1)!}{(q-p)!}=\prod_{j=1}^{p-1}(q-j)
        \equiv (-1)^{p-1}(p-1)!\equiv -1\pmod p
\]
by Wilson's theorem.  Therefore
\[
        B_p[x^p]\equiv \gamma_pa_1^p\equiv (-1)(-1)^p\equiv 1\pmod p,
\]
and $a_p\equiv-1\pmod p$.
\end{proof}

\begin{proof}[Proof of Theorem~\ref{thm:a-zero}]
We argue by induction on $M$.  The case $M=1$ is Lemma~\ref{lem:first-a1-ap}.  Assume $M\ge2$ and
\[
        a_{pi}\equiv (-1)^i\pmod p\qquad(1\le i<M).
\]
By Propositions~\ref{prop:A-zero} and~\ref{prop:C-global},
\[
        a_{pM}\equiv -B_{pM}[x^{pM}]\pmod p.
\]
Since $pM>p$, Proposition~\ref{prop:B-unit-divisible} shows that only indices divisible by $p$ can survive modulo $p$.  With $y=x^p$ we get
\begin{equation}\label{eq:B-div-special}
        B_{pM}[x^{pM}]
        \equiv \frac{(pM)!}{q}[y^M]
        \left(\sum_{i=0}^{M-1}\frac{a_{pi}}{(pi)!}y^i\right)^q
        \pmod p.
\end{equation}
Put
\[
        \nu_i:=\frac{(pi)!}{p^ii!}.
\]
Then $\nu_i\in\Zp^\times$ and
\begin{equation}\label{eq:nu-i}
        \nu_i=\prod_{r=1}^i\prod_{t=1}^{p-1}(p(r-1)+t)\equiv ((p-1)!)^i\equiv (-1)^i\pmod p.
\end{equation}
Thus $a_{pi}\nu_i^{-1}\equiv1\pmod p$ for $0\le i<M$.  Write $a_{pi}\nu_i^{-1}=1+pz_i$ with $z_i\in\Zp$ and define
\[
        H_M(y):=\sum_{i=0}^{M-1}\frac{y^i}{p^ii!},
        \qquad
        L_M(y):=\sum_{i=0}^{M-1}z_i\frac{y^i}{p^ii!}.
\]
Then
\[
        \sum_{i=0}^{M-1}\frac{a_{pi}}{(pi)!}y^i=H_M(y)+pL_M(y).
\]
We claim that
\begin{equation}\label{eq:HM-replacement}
        \frac{p^MM!}{q}[y^M](H_M+pL_M)^q
        \equiv
        \frac{p^MM!}{q}[y^M]H_M^q\pmod p.
\end{equation}
Indeed,
\[
        (H_M+pL_M)^q-H_M^q
        =\sum_{j=1}^q\binom qj p^jH_M^{q-j}L_M^j.
\]
Write
\[
        H_M(y)=\sum_{i=0}^{M-1}h_i y^i,
        \qquad
        L_M(y)=\sum_{i=0}^{M-1}\ell_i y^i,
\]
where
\[
        h_i=\frac1{p^ii!},
        \qquad
        \ell_i=z_i\frac1{p^ii!}.
\]
A typical term in $M![y^M]H_M^{q-j}L_M^j$ is
\[
        \binom{q-j}{a_0,\ldots,a_{M-1}}
        \binom{j}{b_0,\ldots,b_{M-1}}
        \frac{M!}{\prod_{i=0}^{M-1}(i!)^{a_i+b_i}}
        \frac{\prod_{i=0}^{M-1}z_i^{\,b_i}}{p^{\sum_i i(a_i+b_i)}},
\]
where
\[
        \sum_i a_i=q-j,
        \qquad
        \sum_i b_i=j,
        \qquad
        \sum_i i(a_i+b_i)=M.
\]
The multinomial factor $M!/\prod_i (i!)^{a_i+b_i}$ is an integer, and the last displayed condition gives the exact $p$-power denominator $p^M$.  Hence
\[
        M![y^M]H_M^{q-j}L_M^j\in p^{-M}\Zp.
\]
Therefore it is enough to show
\[
        \vp\left(\frac{p^j}{q}\binom qj\right)\ge1\qquad(1\le j\le q).
\]
Using $\binom qj=\frac qj\binom{q-1}{j-1}$, we get
\[
        \frac{p^j}{q}\binom qj=\frac{p^j}{j}\binom{q-1}{j-1},
\]
so
\[
        \vp\left(\frac{p^j}{q}\binom qj\right)\ge j-\vp(j)\ge1.
\]
This gives \eqref{eq:HM-replacement}.

Combining \eqref{eq:B-div-special} and \eqref{eq:HM-replacement}, and using $(pM)!=p^MM!\nu_M$, gives
\[
        a_{pM}\equiv -\nu_M\frac{p^MM!}{q}[y^M]H_M(y)^q\pmod p.
\]
By Lemma~\ref{lem:trunc-exp},
\[
        [y^M]H_M(y)^q=\frac{p^{(e-1)M}}{M!}-\frac{q}{p^MM!}.
\]
Therefore
\[
        a_{pM}\equiv -\nu_M\bigl(p^{e(M-1)}-1\bigr)\pmod p.
\]
Since $M\ge2$ and $e\ge2$, this is congruent to $\nu_M$.  Finally \eqref{eq:nu-i} gives $\nu_M\equiv(-1)^M\pmod p$.
\end{proof}

\section{Vector partitions and the \texorpdfstring{$B$}{B}-term for \texorpdfstring{$a\ne0$}{a not equal 0}}\label{sec:vector-partitions}
The higher base-$p$ digits of the index will be encoded by vectors.

\begin{definition}\label{def:digit-vector}
Fix $L\ge1$.  A digit vector is an element
\[
        d=(d_1,\ldots,d_L)\in\{0,1,\ldots,p-1\}^L.
\]
Its weight and numeric value are
\[
        |d|:=d_1+\cdots+d_L,
        \qquad
        N(d):=\sum_{i=1}^Ld_ip^i.
\]
A vector partition of $d$ is a formal product
\[
        \Pi=\prod_{\beta\ne0}\beta^{r_\beta},
        \qquad r_\beta\in\mathbb Z_{\ge0},
        \qquad \sum_{\beta\ne0}r_\beta\beta=d.
\]
The total number of blocks is $R(\Pi):=\sum_{\beta\ne0}r_\beta$.
\end{definition}

Fix $a\in\{0,1,\ldots,p-1\}$.  For a digit vector $\beta$ define
\[
        H_\beta(y):=\sum_{j=0}^a\frac{a_{N(\beta)+j}}{j!}y^j,
        \quad
        F_a(y):=\sum_{j=0}^a\frac{a_j}{j!}y^j,
        \quad
        \widetilde H_d(y):=\sum_{j=0}^{a-1}\frac{a_{N(d)+j}}{j!}y^j.
\]
Write
\[
        d!:=\prod_{i=1}^Ld_i!,
        \qquad
        \beta!:=\prod_{i=1}^L\beta_i!.
\]

\paragraph{A falling-product identity.}
For an integer parameter $t$ and $N\ge1$, set
\begin{equation}\label{eq:Xi-def}
        \Xi_t(N):=\prod_{u=1}^{N-1}(t-u).
\end{equation}
In particular, when $t=q=p^e$ one has
\begin{equation}\label{eq:Xi-q-modp}
        \Xi_q(N)=\Xi_{p^e}(N)\equiv (-1)^{N-1}(N-1)!\pmod p.
\end{equation}
If $N\le q$, then $\Xi_q(N)=\frac{(q-1)!}{(q-N)!}$.  Hence if $N>q$, then any formal vector-partition contribution with $N$ positive blocks is $0\pmod p$: no actual multinomial term occurs, because the zero-index multiplicity would be negative, and \eqref{eq:Xi-q-modp} shows that the formal continuation is divisible by $p$.  Moreover, in a vector partition of a fixed digit vector $d$, every contributing block multiplicity $r_\beta$ satisfies $0\le r_\beta\le p-1$; since each digit of every block $\beta$ also lies in $\{0,\ldots,p-1\}$, the factorials $r_\beta!$ and $(\beta!)^{r_\beta}$ are $p$-adic units.  Later, when the zero-vector multiplicities $n_j$ appear, one has $\sum_{j=1}^a jn_j\le a\le p-1$, hence each $n_j\le p-1$ and each $n_j!$ is also a $p$-adic unit.  Thus the $p$-divisibility of the formal factor $\Xi_q(N)$ cannot be cancelled by any denominator.

\begin{proposition}[Vector-partition expansion of the $B$-term]\label{prop:vector-B}
Let $d\ne0$ be a digit vector and let $k=N(d)+a$ with $1\le a\le p-1$.  Then
\begin{equation}\label{eq:vector-B}
\begin{aligned}
B_k[x^k]\equiv a![y^a]\Biggl(
&\frac{\widetilde H_d(y)}{F_a(y)} +\sum_{\substack{\Pi\vdash d\\ \Pi\ne(d)}}
(-1)^{R(\Pi)-1}
\frac{(R(\Pi)-1)!\,d!}{\prod_{\beta\ne0}r_\beta!(\beta!)^{r_\beta}} \\
&\qquad\qquad\times \prod_{\beta\ne0}\left(\frac{H_\beta(y)}{F_a(y)}\right)^{r_\beta}
\Biggr)\pmod p.
\end{aligned}
\end{equation}
The case $d=0$ is handled separately in the proof of Theorem~\ref{thm:digit-sum}.
\end{proposition}
\begin{proof}
Let $a^\eta=\prod a_n^{\eta_n}$ be a monomial in $B_k[x^k]$.  Since $k\not\equiv0\pmod p$, a unit coefficient must satisfy $c(\eta)+d(\eta)=e$.  If $d(\eta)<e$, then Lemma~\ref{lem:carry-depth} forces every $\eta_n$ to be divisible by $p$, and then $k=\sum n\eta_n$ is divisible by $p$, impossible.  Hence
\[
        d(\eta)=e,
        \qquad
        c(\eta)=0.
\]
Thus every monomial surviving modulo $p$ is weighted carry-free in the addition $\sum n\eta_n=k$.

Write each positive index uniquely as
\[
        n=N(\beta)+j,
        \qquad
        0\le j\le a,
\]
where $\beta$ records the digits in the $p,p^2,\ldots,p^L$ positions.  For $\beta\ne0$ put
\[
        m_{\beta,j}:=\eta_{N(\beta)+j},
        \qquad
        r_\beta:=\sum_{j=0}^a m_{\beta,j},
\]
and for the zero vector put $n_j:=\eta_j$ for $1\le j\le a$.  Carry-freeness is equivalent to
\[
        \sum_{\beta\ne0}r_\beta\beta=d,
        \qquad
        \sum_{\beta\ne0}\sum_{j=0}^a jm_{\beta,j}+\sum_{j=1}^a jn_j=a.
\]
Thus the nonzero vectors form a vector partition $\Pi=\prod\beta^{r_\beta}$ of $d$.  Since
\[
        \sum_{j=1}^a jn_j\le a\le p-1,
\]
one has $n_j\le p-1$ for every $j$, so the zero-vector factorials $n_j!$ are also $p$-adic units.

For $\Pi=(d)$, only the top choice $j=a$, namely the block $N(d)+a=k$, is forbidden; the choices $0\le j<a$ remain, which is exactly the replacement $H_d\mapsto\widetilde H_d$.  Fix $\Pi$, and set
\[
        R:=R(\Pi),
        \qquad
        t:=\sum_{j=1}^a n_j,
        \qquad
        M:=R+t.
\]
For an actual monomial one has $M\le q$, so $\eta_0=q-M\ge0$.  For the formal extension to $M>q$, we use \eqref{eq:Xi-q-modp}.  Thus the same displayed scalar expression may be written uniformly as
\[
\begin{aligned}
        k!\frac{\Xi_q(M)}{\prod_{\beta,j}m_{\beta,j}!\prod_{j=1}^a n_j!}
        \prod_{\beta,j}\frac1{(N(\beta)+j)!^{m_{\beta,j}}}
        \prod_{j=1}^a\frac1{(j!)^{n_j}}.
\end{aligned}
\]
Modulo $p$,
\[
        \Xi_q(M)=\prod_{u=1}^{M-1}(q-u)\equiv(-1)^{M-1}(M-1)!.
\]
By \eqref{eq:Xi-q-modp}, the same formal expression may be used uniformly even when $M>q$.
By Lucas' theorem in multinomial form, applied to the carry-free addition of the positive indices,
\[
        \frac{k!}{\prod_{\beta,j}(N(\beta)+j)!^{m_{\beta,j}}\prod_j(j!)^{n_j}}
        \equiv
        \frac{a!d!}{\prod_{\beta,j}j!^{m_{\beta,j}}(\beta!)^{m_{\beta,j}}\prod_j(j!)^{n_j}}
        \pmod p.
\]
Summing first over the zero-vector blocks gives the negative-binomial expansion of $F_a(y)^{-R}$, and summing over $m_{\beta,j}$ with fixed $r_\beta$ gives $H_\beta(y)^{r_\beta}/r_\beta!$.  Collecting the factors gives exactly \eqref{eq:vector-B}.
\end{proof}

\section{Low coefficients and a multivariate cumulant collapse for \texorpdfstring{$a\ne0$}{a not equal 0}}\label{sec:cumulant}
The first block of coefficients is independent of $e$.

\begin{proposition}[The first block modulo $p$]\label{prop:first-block}
For $0\le n\le p-1$,
\[
        a_n\equiv (-1)^n(n+1)^{n-1}\pmod p,
\]
with the convention that the case $n=0$ gives $a_0=1$.
\end{proposition}
\begin{proof}
Let $g(x)=\sum_{n\ge0}a_nx^n/n!$, so $f(x)=xg(x)$ and $g(0)=1$.  The B\"ottcher equation is
\[
        g(x^q)=g(x)^q(1+qxg(x)).
\]
Taking logarithms gives
\[
        \log g(x^q)=q\log g(x)+\log(1+qxg(x)).
\]
For $1\le n\le p-1$, the left side has zero $x^n$-coefficient.  Divide the coefficient relation by $q$.  Since
\[
        \frac1q\log(1+qxg(x))=xg(x)+\sum_{m\ge2}(-1)^{m-1}\frac{q^{m-1}}{m}x^mg(x)^m
\]
and $m\le n<p$ in the relevant terms, all summands with $m\ge2$ vanish modulo $p$.  Hence
\[
        [x^n]\log g(x)\equiv-[x^{n-1}]g(x)\pmod p.
\]
Let $F_{p-1}(x)=\sum_{n=0}^{p-1}a_nx^n/n!$.  We get
\begin{equation}\label{eq:first-block-log}
        \log F_{p-1}(x)\equiv -xF_{p-1}(x)\pmod{(p,x^p)}.
\end{equation}
Set $U=xF_{p-1}(x)$.  Then
\[
        U\equiv xe^{-U}\pmod{(p,x^{p+1})}.
\]
This congruence determines the coefficients of $U$ recursively up to degree $p$: the right-hand side has linear term $x$, and the coefficient of $x^m$ depends only on the lower coefficients of $U$.  Let $\widetilde U\in x\Qp[[x]]$ be the characteristic-zero solution of
\[
        \widetilde U=xe^{-\widetilde U}.
\]
Then $U\equiv \widetilde U\pmod{(p,x^{p+1})}$.  By Lagrange inversion over $\Qp$ \cite{Gessel,SuryaWarnke},
\[
        [x^m]\widetilde U=\frac1m[u^{m-1}]e^{-mu}=\frac{(-m)^{m-1}}{m!}\qquad(m\ge1).
\]
For $1\le m\le p-1$ this is $p$-integral.  For $m=p$, the coefficient is $p^{p-1}/p!=p^{p-2}/(p-1)!\in p\Zp$, so it vanishes modulo $p$.  Dividing by $x$ gives the desired formula for $a_n/n!$ for $0\le n\le p-1$.
\end{proof}

\begin{corollary}\label{cor:log-Fa}
For $1\le a\le p-1$,
\[
        \log F_a(y)\equiv -yF_a(y)\pmod{(p,y^{a+1})},
\]
and therefore
\[
        a![y^a]\log F_a(y)\equiv -a\,a_{a-1}\pmod p.
\]
\end{corollary}
\begin{proof}
Since $F_a(y)-F_{p-1}(y)=\bigo(y^{a+1})$, \eqref{eq:first-block-log} gives the first congruence after truncation.  Taking the $y^a$ coefficient gives
\[
        [y^a]\log F_a(y)\equiv-[y^{a-1}]F_a(y)=-\frac{a_{a-1}}{(a-1)!}\pmod p.
\]
\end{proof}

\begin{lemma}[Induction-to-block transfer]\label{lem:block-transfer}
Assume that \eqref{eq:digit-formula-aa} is known for every index $N(\beta)+j$ satisfying either
\[
        0\le j<a\ \text{ and }\ |\beta|\le s,
        \qquad\text{or}\qquad
        j=a\ \text{ and }\ |\beta|<s.
\]
Then, for every nonzero digit vector $\beta$ with $|\beta|<s$,
\[
        H_\beta(y)\equiv (-D)^{|\beta|}F_a(y)\pmod p,
        \qquad
        D:=1+y\frac{d}{dy}.
\]
Consequently, if $d$ has weight $s$, then
\[
        \widetilde H_d(y)\equiv (-D)^sF_a(y)+(-1)^{s-1}(a+1)^s a_a\frac{y^a}{a!}\pmod p.
\]
\end{lemma}
\begin{proof}
For $|\beta|<s$, the displayed hypothesis gives
\[
        a_{N(\beta)+j}\equiv (-1)^{|\beta|}(j+1)^{|\beta|}a_j\pmod p
        \qquad(0\le j\le a).
\]
Summing over $j$ gives the first assertion because $D^mF_a(y)=\sum_{j=0}^a(j+1)^ma_jy^j/j!$.

Now let $d$ have weight $s$.  For $0\le j<a$, the first part of the hypothesis applies with $|d|=s$ and gives
\[
        a_{N(d)+j}\equiv (-1)^s(j+1)^sa_j\pmod p.
\]
Summing only over $0\le j<a$ therefore yields the displayed formula for $\widetilde H_d(y)$; the missing top term $j=a$ contributes exactly
\[
        (-1)^{s-1}(a+1)^sa_a\frac{y^a}{a!}.
\]
\end{proof}

For nonzero digit vectors define
\[
        U_\beta(y):=(-1)^{|\beta|}\frac{D^{|\beta|}F_a(y)}{F_a(y)}.
\]
Introduce variables $t=(t_1,\ldots,t_L)$ and write $t^\beta=\prod_i t_i^{\beta_i}$.  Define cumulants $K_d(U)$ by
\begin{equation}\label{eq:cumulants}
        \log\left(1+\sum_{\beta\ne0}U_\beta(y)\frac{t^\beta}{\beta!}\right)
        =\sum_{d\ne0}K_d(U)\frac{t^d}{d!}.
\end{equation}

\begin{proposition}[Collapse to one variable]\label{prop:cumulant-collapse}
Let $d$ be a nonzero digit vector and set $s=|d|$.  Then
\[
        a![y^a]K_d(U)=(-1)^s a![y^a]N^s\log F_a(y),
        \qquad
        N:=y\frac{d}{dy}.
\]
Consequently,
\[
        a![y^a]K_d(U)\equiv (-1)^{s+1}a^{s+1}a_{a-1}\pmod p.
\]
\end{proposition}
\begin{proof}
To compute the coefficient of the fixed monomial $t^d$, we may enlarge the block set from nonzero digit vectors $\beta\in\{0,\ldots,p-1\}^L$ to all nonzero $\beta\in\mathbb Z_{\ge0}^L$, because any term with some $\beta_i>d_i$ cannot contribute to $t^d$.  Since $0\le d_i\le p-1$, this enlargement does not change the coefficient under consideration.

After this harmless enlargement, $U_\beta$ depends only on $|\beta|$, so the series in \eqref{eq:cumulants} depends only on $T=t_1+\cdots+t_L$:
\[
        1+\sum_{\beta\ne0}U_\beta\frac{t^\beta}{\beta!}
        =\sum_{m\ge0}(-1)^m\frac{D^mF_a(y)}{F_a(y)}\frac{T^m}{m!}
        =\frac{e^{-TD}F_a(y)}{F_a(y)}.
\]
Since $D=1+N$ and $e^{-TN}F_a(y)=F_a(ye^{-T})$, the logarithm is
\[
        -T+\log F_a(ye^{-T})-\log F_a(y).
\]
The coefficient of $t^d/d!$ in a series depending only on $T$ is the coefficient of $T^{|d|}/|d|!$.  Hence the first formula follows.  The linear term $-T$ has no $y^a$ coefficient.  Finally $[y^a]N^s\log F_a=a^s[y^a]\log F_a$, and Corollary~\ref{cor:log-Fa} gives the congruence.
\end{proof}

\section{Proof of \texorpdfstring{Theorems~\ref{thm:digit-sum} and~\ref{thm:SS-special}}{Theorems 1.1 and 1.2}}\label{sec:proof-main-special}
\begin{proof}[Proof of \eqref{eq:digit-formula-aa}]
The case $a=0$ is Theorem~\ref{thm:a-zero}.  For $1\le a\le p-1$ we argue by outer induction on $a$ and inner induction on the higher digit weight $s$.

Fix $a$ and assume the theorem is known for smaller residue classes.  If $s=0$, then $k=a$ and the claim is tautological.  Let $s\ge1$, let $d$ be a digit vector of weight $s$, and put $k=N(d)+a$.  By Proposition~\ref{prop:C-global} and the already known formula in the residue class $a-1$,
\begin{equation}\label{eq:C-in-main}
        C_k[x^k]\equiv a\,a_{k-1}\equiv (-1)^s a^{s+1}a_{a-1}\pmod p.
\end{equation}
By Proposition~\ref{prop:A-zero},
\begin{equation}\label{eq:A-in-main}
        A_k[x^k]\equiv0\pmod p.
\end{equation}

For the $B$-term, put $V_\beta(y)=H_\beta(y)/F_a(y)$, and let $K_d(V)$ be defined by \eqref{eq:cumulants} with $U_\beta$ replaced by $V_\beta$.  Since the one-block partition contributes $V_d$, Proposition~\ref{prop:vector-B} gives
\begin{equation}\label{eq:B-cumulant-main}
        B_k[x^k]\equiv a![y^a]\left(\frac{\widetilde H_d(y)}{F_a(y)}+K_d(V)-V_d(y)\right)\pmod p.
\end{equation}
Every block in a proper partition of $d$ has weight smaller than $s$.  Therefore Lemma~\ref{lem:block-transfer}, applied with the outer induction hypothesis on the residue class and the inner induction hypothesis on the higher digit weight, allows us to replace $V_\beta$ by $U_\beta$ in the proper-partition contribution:
\[
        a![y^a](K_d(V)-V_d)\equiv a![y^a](K_d(U)-U_d)\pmod p.
\]
Moreover Lemma~\ref{lem:block-transfer} gives
\[
        a![y^a]\left(\frac{\widetilde H_d(y)}{F_a(y)}-U_d(y)\right)
        \equiv (-1)^{s-1}(a+1)^s a_a\pmod p.
\]
Together with Proposition~\ref{prop:cumulant-collapse}, this yields
\begin{equation}\label{eq:B-main-special}
        B_k[x^k]\equiv
        (-1)^{s-1}(a+1)^sa_a+(-1)^{s+1}a^{s+1}a_{a-1}\pmod p.
\end{equation}
Substituting \eqref{eq:C-in-main}, \eqref{eq:A-in-main}, and \eqref{eq:B-main-special} into \eqref{eq:ABC}, we get
\[
\begin{aligned}
        a_k
        &\equiv-\bigl((-1)^{s-1}(a+1)^sa_a+(-1)^{s+1}a^{s+1}a_{a-1}\bigr)
        -(-1)^s a^{s+1}a_{a-1}       \\
        &\equiv (-1)^s(a+1)^sa_a\pmod p.
\end{aligned}
\]
The two inductions are complete.
\end{proof}

\begin{proof}[Proof of \eqref{eq:closed-digit-formula}]
If $a=0$, then \eqref{eq:digit-formula-aa} gives $a_k\equiv(-1)^s\pmod p$, which is \eqref{eq:closed-digit-formula}.  If $1\le a\le p-1$, then \eqref{eq:digit-formula-aa} and Proposition~\ref{prop:first-block} give
\[
        a_k\equiv (-1)^s(a+1)^sa_a
        \equiv (-1)^{s+a}(a+1)^{s+a-1}\pmod p.
\]
\end{proof}

\begin{proof}[Proof of Theorem~\ref{thm:SS-special}]
The congruence $a_{pm}\equiv(-1)^m\pmod p$ is Theorem~\ref{thm:a-zero}.  For the second congruence, write
\[
        pm-1=(m-1)p+(p-1),
        \qquad
        s:=\Sp(m-1).
\]
If $m=1$, Proposition~\ref{prop:first-block} gives $a_{p-1}\equiv0\pmod p$.  If $m\ge2$, then $s\ge1$, and \eqref{eq:digit-formula-aa} gives
\[
        a_{pm-1}\equiv(-1)^sp^sa_{p-1}\equiv0\pmod p.
\]
For the third congruence, write
\[
        pm-2=(m-1)p+(p-2),
        \qquad
        s:=\Sp(m-1).
\]
Then \eqref{eq:digit-formula-aa} gives
\[
        a_{pm-2}\equiv(-1)^s(p-1)^sa_{p-2}\equiv a_{p-2}\pmod p,
\]
and Proposition~\ref{prop:first-block} gives $a_{p-2}\equiv-1\pmod p$.
\end{proof}

\section{The family \texorpdfstring{$\varphi_{r,e}$}{phi r e} for \texorpdfstring{$r\ge1$}{r >= 1}}\label{sec:higher-recursion}
Fix $r\ge1$ and $e\ge2$, and put $q=p^e$.  We now write $a_k(r)=a_k(r,e)$.  Since
\[
        \varphi_{r,e}(x)=x^q+qp^rx^{q+1}\in x^q+qx^{q+1}\Zp[\![x]\!],
\]
\cite[Theorem~3(b)]{SalernoSilverman} gives $a_k(r)\in\Zp$ for every $k$.  Comparing the coefficient of $x^{q+1}$ in the B\"ottcher equation
\[
        \varphi_{r,e}\bigl(f_{r,e}(x)\bigr)=f_{r,e}(x^q)
\]
immediately gives
\begin{equation}\label{eq:a1-r}
        a_1(r)=-p^r,
\end{equation}
since the left-hand side contributes $q a_1(r)+q p^r$ to $x^{q+1}$, while the right-hand side has no $x^{q+1}$ term.  The recursion is
\begin{equation}\label{eq:ABC-r}
        a_k(r)=A_k[x^k]-B_k[x^k]-p^rC_k[x^k],
\end{equation}
where $A_k,B_k,C_k$ are the same expressions as in Section~\ref{sec:recursion}, with $a_\ell$ replaced by $a_\ell(r)$.  Thus all scalar coefficient formulas for $A$, $B$, and $C$ remain valid.  In particular, \eqref{eq:B-cd} becomes
\[
        \vp(\gamma_B(\eta))=c(\eta)+d(\eta)-e,
\]
and all $B$-coefficients are $p$-integral by Lemma~\ref{lem:B-integrality}.  If $p\mid k$, every coefficient of $C_k[x^k]$ is divisible by $p$ by Proposition~\ref{prop:C-global}.

We begin with a low-index bound for the initial pure-power levels.

\begin{lemma}[Low-index bound]\label{lem:low-index-pe}
For $r\ge1$, $e\ge2$, and $0\le j\le p-1$,
\[
        \vp\bigl(a_j(r)\bigr)\ge jr.
\]
Consequently, any product of coefficients whose lower indices have total sum $N<p$ has valuation at least $rN$.
\end{lemma}
\begin{proof}
Let
\[
        g_r(x):=\sum_{j\ge0}\frac{a_j(r)}{j!}x^j,
\]
so that $f_{r,e}(x)=xg_r(x)$ and $g_r(0)=1$.  The B\"ottcher equation becomes
\[
        g_r(x^q)=g_r(x)^q\bigl(1+qp^r xg_r(x)\bigr),\qquad q=p^e.
\]
Taking logarithms gives
\[
        \log g_r(x^q)=q\log g_r(x)+\log\bigl(1+qp^r xg_r(x)\bigr).
\]
Fix $1\le n\le p-1$.  The coefficient of $x^n$ on the left side is $0$.  Dividing the coefficient relation by $q$ therefore gives
\[
        [x^n]\log g_r(x)=-p^r[x^{n-1}]g_r(x)+E_n,
\]
where $E_n$ is a sum of terms coming from
\[
        \sum_{m\ge2}(-1)^{m-1}\frac{q^{m-1}p^{mr}}{m}x^m g_r(x)^m.
\]
We prove the bound by induction on $n$.  The case $n=1$ is \eqref{eq:a1-r}.  Assume it for all smaller indices.  A monomial contributing to the coefficient of $x^{n-m}$ in $g_r(x)^m$ is a product of lower coefficients whose total index is $n-m$, so by the induction hypothesis it has valuation at least $r(n-m)$.  Since $m\le n<p$, the denominator $m$ is a $p$-adic unit.  Therefore every summand of $E_n$ has valuation at least
\[
        (m-1)e+mr+r(n-m)=rn+(m-1)e\ge rn+2.
\]
On the other hand, the coefficient $[x^n]\log g_r(x)$ is $a_n(r)/n!$ plus a polynomial in $a_1(r),\ldots,a_{n-1}(r)$.  Each monomial in that polynomial has total lower index $n$, hence valuation at least $rn$.  Therefore the displayed relation gives
\[
        \vp\!\left(\frac{a_n(r)}{n!}\right)\ge rn.
\]
Since $n!\in\Zp^\times$ for $n<p$, the claim follows.
\end{proof}

\begin{lemma}[The first two pure powers]\label{lem:first-two-pure-pe}
For $r\ge1$ and $e\ge2$,
\[
        v_0(r,e)=r,
        \qquad
        v_1(r,e)=pr,
\]
and
\[
        p^{-v_i(r,e)}a_{p^i}(r,e)\equiv -1\pmod p
        \qquad (i=0,1).
\]
\end{lemma}
\begin{proof}
The case $i=0$ is \eqref{eq:a1-r}.  For $i=1$, take $k=p$ in \eqref{eq:ABC-r}.  The $A$-term is $0$ because $q\nmid p$.  By Proposition~\ref{prop:C-global} and Lemma~\ref{lem:low-index-pe}, the term $p^rC_p[x^p]$ has valuation at least $pr+1$: every monomial coefficient of $C_p[x^p]$ is divisible by $p$, and its monomials are products of lower coefficients with total lower index $p-1$.

By Proposition~\ref{prop:B-unit-divisible}, the unique unit-coefficient monomial in $B_p[x^p]$ is $a_0(r)^{q-p}a_1(r)^p$.  Its scalar coefficient is
\[
\begin{aligned}
        \gamma_{1,e}
        &=\frac{p!}{q}\binom{q}{q-p,p}
         =\frac{(q-1)!}{(q-p)!} \\
        &=\prod_{j=1}^{p-1}(q-j)
         \equiv (-1)^{p-1}(p-1)!
         \equiv -1\pmod p.
\end{aligned}
\]
by Wilson's theorem.  All other $B$-monomials have scalar coefficient in $p\Zp$; by Lemma~\ref{lem:low-index-pe} their coefficient-products have valuation at least $pr$, so they contribute only $O(p^{pr+1})$.  Therefore
\[
        a_p(r)\equiv -\gamma_{1,e}a_1(r)^p\equiv -p^{pr}\pmod{p^{pr+1}}.
\]
This gives the case $i=1$.
\end{proof}

\begin{lemma}\label{lem:Tp-power-pe}
For every $m\ge0$,
\[
        \Tp(p^m)\equiv (-1)^m\pmod p.
\]
\end{lemma}
\begin{proof}
The assertion is clear for $m=0$.  For $m\ge1$, group the integers from $1$ to $p^m$ according to their $p$-adic valuation.  The prime-to-$p$ parts of the numbers with valuation $j$ contribute, modulo $p$, the product of all nonzero residues modulo $p$, repeated $p^{m-j-1}$ times.  Thus
\[
        \Tp(p^m)\equiv \prod_{j=0}^{m-1}((p-1)!)^{p^{m-j-1}}
        \equiv \prod_{j=0}^{m-1}(-1)^{p^{m-j-1}}
        \equiv (-1)^m\pmod p,
\]
using Wilson's theorem and the fact that $p$ is odd.
\end{proof}

We now isolate the pure-power part of the $B$-term.

\begin{lemma}[Pure-power unit $B$-monomial]\label{lem:pure-power-B-unit}
Let $n\ge1$.  In $B_{p^n}[x^{p^n}]$, the unique monomial with $p$-adic unit scalar coefficient is
\[
        a_0(r)^{q-p}a_{p^{n-1}}(r)^p.
\]
Its scalar coefficient is $\gamma_{n,e}$ and satisfies $\gamma_{n,e}\in\Zp^\times$ and $\gamma_{n,e}\equiv-1\pmod p$.
\end{lemma}
\begin{proof}
Let a unit-coefficient monomial have multiplicities $\eta_i$.  Then $c(\eta)+d(\eta)=e$.  If $d(\eta)=e$, then $c(\eta)=0$, so the weighted addition $\sum_i i\eta_i=p^n$ is carry-free.  Since the target is a pure power, the positive part would then consist of a single index $p^n$ with multiplicity one, but every index occurring in $B_{p^n}[x^{p^n}]$ is $<p^n$.  Hence $d(\eta)<e$.

Put $s=e-d(\eta)\ge1$.  By Lemma~\ref{lem:carry-depth}, we may write $\eta_i=p^s u_i$.  Let
\[
        M:=\sum_i i u_i,
        \qquad
        A:=\sum_i u_i\Sp(i).
\]
Since $c(\eta)=s$, equation \eqref{eq:B-cd} becomes
\[
        p^sA-\Sp(M)=s(p-1).
\]
Because $\Sp(M)\le A$, we obtain
\[
        (p^s-1)A\le s(p-1).
\]
For $s\ge2$ this is impossible, because $p^s-1>s(p-1)$ for odd $p$.  Hence $s=1$.  The same inequality then forces $A=1$ and $\Sp(M)=1$.  Thus exactly one positive index occurs after division by $p$, with multiplicity one and digit sum one.  Since
\[
        p^n=\sum_i i\eta_i=p\sum_i i u_i=pM,
\]
that index must be $M=p^{n-1}$.  Therefore $\eta_{p^{n-1}}=p$ and $\eta_0=q-p$.

The scalar coefficient is
\[
        \gamma_{n,e}
        =\frac{(p^n)!}{p\bigl((p^{n-1})!\bigr)^p}\binom{p^e-1}{p-1}.
\]
The first factor is a $p$-adic unit, and by Lemma~\ref{lem:Tp-power-pe}
\[
        \frac{(p^n)!}{p\bigl((p^{n-1})!\bigr)^p}
        \equiv \frac{\Tp(p^n)}{\Tp(p^{n-1})^p}
        \equiv \frac{(-1)^n}{((-1)^{n-1})^p}
        \equiv -1\pmod p.
\]
Lucas' theorem gives $\binom{p^e-1}{p-1}\equiv1\pmod p$, so $\gamma_{n,e}\equiv-1\pmod p$.
\end{proof}

\begin{proposition}[Exact decomposition on pure powers]\label{prop:pure-decomp}
For every $n\ge1$ there exist polynomials $R_n,S_n\in\Zp[a_1(r),\ldots,a_{p^n-1}(r)]$ such that
\begin{equation}\label{eq:pure-decomp-small}
        a_{p^n}(r)=-\gamma_{n,e}a_{p^{n-1}}(r)^p+pR_n+p^{r+1}S_n
        \qquad(1\le n<e),
\end{equation}
and
\begin{equation}\label{eq:pure-decomp-large}
        a_{p^n}(r)=\alpha_{n,e}a_{p^{n-e}}(r)-\gamma_{n,e}a_{p^{n-1}}(r)^p+pR_n+p^{r+1}S_n
        \qquad(n\ge e).
\end{equation}
Moreover,
\begin{itemize}[leftmargin=2em]
\item every monomial of $R_n$ is a product of lower coefficients whose weighted total index is $p^n$;
\item every monomial of $S_n$ is a product of lower coefficients whose weighted total index is $p^n-1$;
\item all scalar coefficients of $R_n$ and $S_n$ are $p$-integral;
\item
\[
        \vp(\alpha_{n,e})=\Delta_{n,e}=\mu_ep^{n-e}-e,
        \qquad
        p^{-\Delta_{n,e}}\alpha_{n,e}\equiv(-1)^e\pmod p.
\]
\end{itemize}
\end{proposition}
\begin{proof}
The $A$-term contributes only when $n\ge e$, in which case
\[
        A_{p^n}[x^{p^n}]=\frac{(p^n)!}{p^e(p^{n-e})!}a_{p^{n-e}}(r)=\alpha_{n,e}a_{p^{n-e}}(r).
\]
Legendre's formula gives the stated valuation.  The unit congruence follows from Lemma~\ref{lem:Tp-power-pe}:
\[
        p^{-\vp(\alpha_{n,e})}\alpha_{n,e}
        \equiv \frac{\Tp(p^n)}{\Tp(p^{n-e})}
        \equiv \frac{(-1)^n}{(-1)^{n-e}}
        \equiv (-1)^e\pmod p.
\]

By Lemma~\ref{lem:pure-power-B-unit}, the unique $B$-monomial with unit scalar coefficient is $a_0(r)^{q-p}a_{p^{n-1}}(r)^p$, and its scalar coefficient is $\gamma_{n,e}$.  Every other $B$-monomial has scalar coefficient in $p\Zp$ by Lemma~\ref{lem:B-integrality}.  Hence the remaining $B$-contribution may be written as $pR_n$, where every monomial of $R_n$ has weighted total index $p^n$.

Likewise, since $p\mid p^n$, Proposition~\ref{prop:C-global} gives
\[
        C_{p^n}[x^{p^n}]\in p\Zp[a_1(r),\ldots,a_{p^n-1}(r)].
\]
Every monomial occurring there has weighted total index $p^n-1$, so after multiplying by the outer factor $p^r$ in \eqref{eq:ABC-r} we obtain the term $p^{r+1}S_n$.  Thus the stated support and integrality properties of $R_n$ and $S_n$ follow.
\end{proof}

\section{Recursive induction for the higher fibers}\label{sec:higher-induction}
We are ready to prove Theorem~\ref{thm:higher-structure}.  Throughout this section $r,e$ are fixed, and we write $v_i=v_i(r,e)$, $\Lambda=\Lambda_{r,e}$, and $m=m_{r,e}$ when no confusion can arise.  By the Section~\ref{sec:higher-recursion} formula \eqref{eq:a1-r}, one has $v_0=r$.

\medskip
\noindent\emph{Outline of the higher-fiber argument.}
We treat the divisible non-pure case in four steps.  First, global $B$-coefficient integrality gives $p$-integral scalar coefficients.  Second, the weight-excess identity controls the filtered degree under the pure-power inequalities already known.  Third, every unit scalar term in the divisible non-pure sector is carry-free.  Finally, the carry-free cumulant coefficient gives the degree-$\Lambda$ initial unit sector.  We use this scheme repeatedly below.

We refer to the four assertions of Theorem~\ref{thm:higher-structure} as (L), (D), (P), and (S).  For $N\ge1$, let $\mathcal I(N)$ be the package consisting of
\begin{enumerate}[label=\textup{(\arabic*)},leftmargin=2.2em]
\item (L) for every $k\le N$;
\item (D) for every divisible non-pure $k\le N$;
\item (P) for every pure power $p^n\le N$ with $n\ge1$;
\item (S) for every finite sum $n_1+\cdots+n_t=M$ with $M\le N$.
\end{enumerate}
We prove $\mathcal I(N)$ for all $N$ by strong induction.  In the step from $\mathcal I(N-1)$ to $\mathcal I(N)$, we first handle the pure-power case $N=p^n$, then the divisible non-pure case, and finally the global lower bound at the general index $N$.  The subadditivity statement is recovered each time from the inequalities $v_{j+1}\le p v_j$.  In this way the dependence between (L), (D), (P), and (S) is explicit, and there is no circularity.

Two auxiliary lemmas will be used repeatedly.

\begin{lemma}[Top-digit lower bound from lower levels]\label{lem:top-digit-lower-pe}
Assume that
\[
        v_{i+1}\le p v_i\qquad (0\le i\le n-2).
\]
Let $e_j$ be nonnegative integers with $0\le j<p^n$ and
\[
        \sum_{j=0}^{p^n-1}e_jj=p^n.
\]
Then
\[
        \sum_{j=0}^{p^n-1}e_j\Lambda(j)\ge p\,v_{n-1}.
\]
\end{lemma}
\begin{proof}
Write every $j$ in base $p$ and regard the left side as the total digit weight of the multiset consisting of $e_j$ copies of $j$.  Passing from this multiset to the base-$p$ representation of the total $p^n$ amounts to repeatedly replacing $p$ copies of $p^i$ by one copy of $p^{i+1}$.  We perform all carries below level $n-1$ but do not carry the resulting $p$ copies of $p^{n-1}$ to level $n$.  Each carry weakly decreases the total weight because $v_{i+1}\le pv_i$.  Therefore the initial weight is at least the final weight $p v_{n-1}$.
\end{proof}

\begin{lemma}[Weight excess from carries in degree $p^e$]\label{lem:weight-excess-pe}
Fix $M\ge1$ and write
\[
        M=M_0+M_1p+\cdots+M_sp^s,\qquad 0\le M_i\le p-1.
\]
Let $(e_i)_{0\le i<M}$ satisfy
\[
        \sum_{i=0}^{M-1}e_i=p^e,
        \qquad
        \sum_{i=0}^{M-1}ie_i=M.
\]
Write $i=\sum_{j\ge0}i_jp^j$ and set
\[
        u_j:=\sum_{i=0}^{M-1}e_i i_j.
\]
Let $c_j\ge0$ be the carry numbers determined by
\[
        u_j=M_j+p c_j-c_{j-1}\qquad (j\ge0),\qquad c_{-1}:=0.
\]
Then
\begin{equation}\label{eq:weight-excess-pe}
        \sum_{j\ge0}u_jv_{j+1}-\Lambda(pM)
        =\sum_{j\ge0}c_j\bigl(pv_{j+1}-v_{j+2}\bigr).
\end{equation}
If, in addition,
\[
        v_{j+2}\le pv_{j+1}
        \qquad\text{for every }j\text{ with }c_j>0,
\]
then
\[
        \sum_{j\ge0}u_jv_{j+1}\ge \Lambda(pM),
\]
and equality holds if and only if every nonzero carry $c_j$ occurs at a level with $v_{j+2}=pv_{j+1}$.
\end{lemma}
\begin{proof}
Since the base-$p$ digits of $pM$ are $0,M_0,M_1,\ldots,M_s$, one has
\[
        \Lambda(pM)=\sum_{j\ge0}M_jv_{j+1}.
\]
Subtracting this from $\sum_j u_jv_{j+1}$ and using the defining relation for the carries gives
\[
\begin{aligned}
        \sum_{j\ge0}u_jv_{j+1}-\Lambda(pM)
        &=\sum_{j\ge0}(pc_j-c_{j-1})v_{j+1}\\
        &=\sum_{j\ge0}c_j\bigl(pv_{j+1}-v_{j+2}\bigr),
\end{aligned}
\]
which is \eqref{eq:weight-excess-pe}.  Under the displayed hypothesis every summand on the right is nonnegative, and the final assertion is immediate.
\end{proof}

\begin{lemma}[Carry-free coefficient in degree $p^e$]\label{lem:carry-free-coeff-pe}
Let $M\ge1$ be not a power of $p$, and write $M=M_0+M_1p+\cdots+M_sp^s$.  Introduce variables $Y_0,\ldots,Y_s$ and put
\[
        Y^{[i]}:=\prod_jY_j^{i_j}\quad\text{if }i=\sum_ji_jp^j,
        \qquad Y^{[0]}:=1.
\]
Then the coefficient of $Y^{[M]}$ in
\[
        \frac{(pM)!}{p^e}[y^M]
        \left(\sum_{i=0}^{M-1}\frac{Y^{[i]}}{(pi)!}y^i\right)^{p^e}
\]
is congruent to $-1$ modulo $p$.
\end{lemma}
\begin{proof}
A contributing multi-index has multiplicities $\eta_i$ satisfying $\sum_i\eta_i=p^e$, $\sum_i i\eta_i=M$, and the condition on the $Y$-exponent says that this addition is digitwise carry-free and has digit vector $(M_0,\ldots,M_s)$.

Set $\nu_i=(pi)!/(p^ii!)$.  As in \eqref{eq:nu-i}, $\nu_i\equiv (-1)^i\pmod p$.  Since $\sum_i i\eta_i=M$,
\[
        \frac{\nu_M}{\prod_i\nu_i^{\eta_i}}
        \equiv (-1)^{M-\sum_i i\eta_i}
        \equiv 1\pmod p,
\]
so the signs from $\nu_M$ and the $\nu_i$ cancel.  For fixed nonzero block multiplicities, the remaining scalar is
\[
        (-1)^{N-1}\frac{(N-1)!\,d!}{\prod_{\beta\ne0}r_\beta!(\beta!)^{r_\beta}},
\]
where $d=(M_0,\ldots,M_s)$ and $N=\sum_{\beta\ne0}r_\beta$.  Indeed, the only place where $e$ enters is the falling product
\[
        \Xi_{p^e}(N)=\prod_{u=1}^{N-1}(p^e-u)\equiv(-1)^{N-1}(N-1)!\pmod p.
\]
By \eqref{eq:Xi-q-modp}, the carry-free vector-partition sum may be taken uniformly without imposing $N\le q$.
Thus the sum over all proper vector partitions of $d$ is the cumulant coefficient in
\[
        \log\left(1+\sum_{\beta\ne0}\frac{t^\beta}{\beta!}\right).
\]
To compute the coefficient of $t^d$, we may temporarily enlarge the block set to all nonzero $\beta\in\mathbb Z_{\ge0}^{s+1}$, because any coordinate $\beta_j>d_j$ cannot contribute to $t^d$.  Then
\[
        1+\sum_{\beta\ne0}\frac{t^\beta}{\beta!}
        =\prod_{j=0}^s\left(\sum_{n\ge0}\frac{t_j^n}{n!}\right)
        =e^{t_0+\cdots+t_s}.
\]
Hence the logarithm is $t_0+\cdots+t_s$.  Because $M$ is not a power of $p$, the digit vector $d$ is not a standard basis vector, so the full cumulant coefficient is $0$.  The one-block partition of $d$ is excluded by the range $0\le i<M$.  Therefore the proper-partition sum is $-1\pmod p$.
\end{proof}

\begin{lemma}[Unit terms are carry-free in the divisible non-pure sector]\label{lem:unit-carryfree-pe}
Let $M\ge1$ be not a power of $p$, and let $\eta=(\eta_i)_{0\le i<M}$ satisfy
\[
        \sum_{i=0}^{M-1}\eta_i=p^e,
        \qquad
        \sum_{i=0}^{M-1}i\eta_i=M.
\]
Consider the corresponding term in
\[
        \frac{(pM)!}{p^e}[y^M]
        \left(\sum_{i=0}^{M-1}\frac{Y^{[i]}}{(pi)!}y^i\right)^{p^e}.
\]
If its scalar coefficient is a $p$-adic unit, then the addition $\sum_i i\eta_i=M$ is carry-free.  Equivalently, the associated monomial in the variables $Y_j$ is already $Y^{[M]}$.
\end{lemma}
\begin{proof}
By \eqref{eq:B-cd}, the unit condition is $c(\eta)+d(\eta)=e$.  Suppose first that $c(\eta)>0$.  Then $d(\eta)<e$.  Put $s=e-d(\eta)\ge1$.  By Lemma~\ref{lem:carry-depth}, every $\eta_i$ is divisible by $p^s$; write $\eta_i=p^su_i$.  Let
\[
        M':=\sum_i iu_i,
        \qquad
        A:=\sum_i u_i\Sp(i).
\]
Then $M=p^sM'$ and the equality $c(\eta)=s$ becomes
\[
        p^sA-\Sp(M')=s(p-1).
\]
Since $\Sp(M')\le A$, we get
\[
        (p^s-1)A\le s(p-1).
\]
For $s\ge2$ this is impossible.  Therefore $s=1$, and the same inequality forces $A=1$ and $\Sp(M')=1$.  Hence $M'$ is a power of $p$, and the unique positive index after division by $p$ is a power of $p$.  Therefore $M=pM'$ is itself a power of $p$, contradicting the hypothesis.  This contradiction shows that $c(\eta)=0$.

Thus the addition $\sum_i i\eta_i=M$ is carry-free.  The exponent vector in the variables $Y_j$ is therefore exactly the base-$p$ digit vector of $M$, so the monomial is $Y^{[M]}$.
\end{proof}

In particular, every unit scalar term in the divisible non-pure sector is carry-free; consequently, the degree-$\Lambda$ initial unit sector is computed by the carry-free cumulant coefficient.

For the associated-graded formalism, let
\[
        \mathcal R_{r,e}:=\Zp[Y_0,Y_1,\ldots],
        \qquad
        \wt(Y_j):=v_{j+1}(r,e).
\]
For an integer $\lambda\in\mathbb Z_{\ge0}$, define
\[
        F^\lambda\mathcal R_{r,e}:=
        \left\{\sum_{\nu} c_\nu Y^\nu:
        \vp(c_\nu)+\sum_{j\ge0}\nu_jv_{j+1}(r,e)\ge \lambda
        \text{ for every }\nu\right\}.
\]
Equivalently, $F^\lambda\mathcal R_{r,e}$ is the $\Zp$-span of the monomials $p^aY^\nu$ with
\[
        a+\sum_{j\ge0}\nu_jv_{j+1}(r,e)\ge \lambda.
\]
We write
\[
        \gr_\Lambda\mathcal R_{r,e}:=\bigoplus_{\lambda\in\mathbb Z_{\ge0}}F^\lambda\mathcal R_{r,e}/F^{\lambda+1}\mathcal R_{r,e},
\]
and denote by
\[
        \pi:=\operatorname{in}(p)\in \gr_\Lambda^1\mathcal R_{r,e}
\]
the initial form of $p$.  Thus $(\gr_\Lambda\mathcal R_{r,e})/(\pi)$ is the $\Fp$-polynomial ring on the variables $Y_j$, and throughout the next proposition we record only the $\pi$-free unit-coefficient part of the initial form.

\begin{proposition}[$\pi$-free initial form of the divisible $B$-term]\label{prop:initial-form-B-pe}
Fix $r\ge1$ and $e\ge2$.  Let $M\ge1$ be not a power of $p$, write
\[
        M=M_0+M_1p+\cdots+M_sp^s,
        \qquad
        0\le M_i\le p-1,
\]
and set $k=pM$.  Assume that
\[
        v_{j+2}(r,e)\le pv_{j+1}(r,e)
        \qquad (0\le j\le s-1).
\]
Define
\[
        Y^{[i]}:=\prod_{j\ge0}Y_j^{i_j}
        \qquad\text{for }i=\sum_{j\ge0} i_jp^j,
\]
so in particular $Y^{[0]}=1$, and set
\[
        \mathcal B_{M,r,e}(Y):=\frac{(pM)!}{p^e}[y^M]
        \left(\sum_{i=0}^{M-1}\frac{Y^{[i]}}{(pi)!}y^i\right)^{p^e}.
\]
Then the degree-$\Lambda(k)$ class of $\mathcal B_{M,r,e}(Y)$ in
$(\gr_\Lambda\mathcal R_{r,e})/(\pi)$ is
\[
        -Y^{[M]}.
\]
\end{proposition}
\begin{proof}
Expand $\mathcal B_{M,r,e}(Y)$ as a sum over multi-indices $\eta=(\eta_i)_{0\le i<M}$ with $\sum_i\eta_i=p^e$ and $\sum_i i\eta_i=M$.  By Lemma~\ref{lem:B-integrality}, every scalar coefficient is $p$-integral.  If the associated exponent vector in the variables $Y_j$ is $u=(u_0,u_1,\ldots)$, then the total sum being $M$ implies that there is no carry out of the $p^s$-place, so only carry levels $0,\ldots,s-1$ can occur.  Lemma~\ref{lem:weight-excess-pe} therefore gives
\[
        \sum_{j\ge0}u_jv_{j+1}(r,e)\ge \Lambda(k),
\]
under the displayed hypothesis on the pure-power inequalities.  Hence every term has filtered weight at least $\Lambda(k)$.

Passing to $(\gr_\Lambda\mathcal R_{r,e})/(\pi)$ kills every term whose scalar coefficient is divisible by $p$.  If a term survives in degree $\Lambda(k)$, then its scalar coefficient is a $p$-adic unit, so Lemma~\ref{lem:unit-carryfree-pe} shows that the underlying addition $\sum_i i\eta_i=M$ is carry-free.  Thus its monomial is already $Y^{[M]}$.

The sum of the scalar coefficients of all such carry-free terms is the coefficient computed in Lemma~\ref{lem:carry-free-coeff-pe}, namely $-1$ modulo $p$.  Therefore the degree-$\Lambda(k)$ class of $\mathcal B_{M,r,e}(Y)$ in $(\gr_\Lambda\mathcal R_{r,e})/(\pi)$ is $-Y^{[M]}$.
\end{proof}

\begin{remark}[The stable-layer quotient]\label{rem:stable-quotient-pe}
If $v_{j+2}(r,e)=pv_{j+1}(r,e)$, then $Y_j^p$ and $Y_{j+1}$ have the same filtered degree, so one may pass further to the quotient with relations $Y_j^p=Y_{j+1}$.  By Proposition~\ref{prop:initial-form-B-pe}, however, the surviving unit part in degree $\Lambda(k)$ is already carry-free.  Thus this extra quotient is not needed in the scalar computation.
\end{remark}

\begin{lemma}[Abstract pure-power branch pattern and no tie]
\label{lem:abstract-branch-pe}
Fix integers $r\ge1$ and $e\ge2$.  Let $(w_n)_{n\ge0}$ be defined by
\[
        w_0=r,
        \qquad
        w_n=pw_{n-1}\quad(1\le n<e),
\]
and, for $n\ge e$,
\[
        w_n=\min\{\Delta_{n,e}+w_{n-e},\,p w_{n-1}\}.
\]
Set
\[
        s:=\left\lceil\frac r e\right\rceil.
\]
Then the branch word for this abstract recursion is
\[
        (B^{e-1}A)^sB^\infty.
\]
Equivalently, the $A$-branch occurs exactly at $n=e,2e,\ldots,se$, every other level is $B$-dominated, and for every $n\ge e$ one has
\[
        \Delta_{n,e}+w_{n-e}\ne p w_{n-1}.
\]
Moreover
\[
        w_{je+t}=p^t\left(r+\frac{p^{j e}-1}{p-1}-e\,j\right)
        \qquad(0\le j\le s,\ 0\le t\le e-1),
\]
and $w_n=p^{n-e s}w_{e s}$ for all $n\ge e s$.
\end{lemma}
\begin{proof}
For $0\le n<e$, only the $B$-branch is present, so $w_n=p^nr$.

Suppose that $A$ has already occurred at the levels $e,2e,\ldots,(j-1)e$.  Then
\[
        w_{(j-1)e}
        =r+\sum_{h=1}^{j-1}(\mu_ep^{(h-1)e}-e)
        =r+\frac{p^{(j-1)e}-1}{p-1}-e(j-1).
\]
Since the next $e-1$ levels after $(j-1)e$ are $B$-dominated, we have $w_{je-1}=p^{e-1}w_{(j-1)e}$.  Hence
\[
\begin{aligned}
        \bigl(\Delta_{je,e}+w_{(j-1)e}\bigr)-p w_{je-1}
        &=\mu_ep^{(j-1)e}-e-(p^e-1)w_{(j-1)e}\\
        &=(p^e-1)\bigl(e(j-1)-r\bigr)+\mu_e-e.
\end{aligned}
\]
Because $0<\mu_e-e<p^e-1$, this quantity is negative exactly when $r\ge e(j-1)+1$.  Therefore the $A$-branch occurs precisely for $j=1,\ldots,s$, and it never ties with the $B$-branch.

If the $A$-branch occurs at $je$, then for $1\le t\le e-1$ one has
\[
\begin{aligned}
        \bigl(\Delta_{je+t,e}+w_{(j-1)e+t}\bigr)-p w_{je+t-1}
        &=\Delta_{je+t,e}+p^t w_{(j-1)e}\\
        &\qquad -p^t\bigl(\Delta_{je,e}+w_{(j-1)e}\bigr)\\
        &=\Delta_{je+t,e}-p^t\Delta_{je,e}\\
        &=e(p^t-1)>0,
\end{aligned}
\]
so the next $e-1$ levels are $B$-dominated.  If the comparison at a level $je$ is $B$-dominated, then the same explicit comparison shows that every later multiple of $e$ is also $B$-dominated.  The intermediate $e-1$ levels are again $B$-dominated by the displayed formula.  This gives the branch word and the explicit formula for $w_{je+t}$.
\end{proof}

\begin{proof}[Proof of Theorem~\ref{thm:higher-structure}]
We verify the induction package $\mathcal I(N)$ by strong induction on $N$.

The base case $\mathcal I(1)$ follows from \eqref{eq:a1-r}: one has
\[
        a_1(r)=-p^r,
        \qquad
        \Lambda(1)=v_0=r,
\]
so (L) holds at $k=1$.  The clauses (D) and (P) are vacuous for $N=1$, and (S) is tautological for totals $M\le1$.

\smallskip
\noindent\emph{Pure powers.}
Let $(w_n)_{n\ge0}$ be the abstract sequence from Lemma~\ref{lem:abstract-branch-pe}.  As part of the induction on $\mathcal I(N)$, whenever a pure-power level $p^j$ has already been treated we also record the equality $v_j=w_j$.

We treat $n=1$ first.  By Lemma~\ref{lem:first-two-pure-pe} we have $v_0=w_0=r$ and $v_1=w_1=pr$.  Now fix $n\ge2$.  Assume (L) and (D) hold for every $k<p^n$, and (P) holds for every pure-power level below $p^n$.  Then (S) is available for all integers whose highest base-$p$ digit is at most $p^{n-1}$: all pure-power levels below $p^n$ have already been treated, so the inequalities $v_{j+1}\le p v_j$ are known for $0\le j\le n-2$, and the carry argument gives subadditivity for every sum supported in those digit places.

We begin with the term $pR_n$ in Proposition~\ref{prop:pure-decomp}.  Every monomial in $R_n$ is a product of lower coefficients whose total index is $p^n$.  By the lower-bound induction hypothesis and Lemma~\ref{lem:top-digit-lower-pe}, its product valuation is at least $p v_{n-1}$.  The extra factor $p$ in $pR_n$ raises the valuation to at least $p v_{n-1}+1$.

For $p^{r+1}S_n$, every monomial in $S_n$ is a product of lower coefficients whose total index is $p^n-1$.  Hence (L) for lower indices and (S) at lower levels show that every monomial contribution to $p^{r+1}S_n$ has valuation at least
\[
        r+1+\Lambda(p^n-1).
\]
Now
\[
        \Lambda(p^n-1)=(p-1)\sum_{i=0}^{n-1}v_i,
\]
because the base-$p$ expansion of $p^n-1$ has all digits equal to $p-1$ below level $n$.  Using $v_{j+1}\le p v_j$ for $0\le j\le n-2$, we obtain
\[
        \sum_{i=0}^{n-2}(p-1)v_i\ge \sum_{i=0}^{n-2}(v_{i+1}-v_i)=v_{n-1}-v_0.
\]
Therefore
\[
\begin{aligned}
        r+1+\Lambda(p^n-1)
        &=1+v_0+(p-1)\sum_{i=0}^{n-1}v_i\\
        &\ge 1+v_0+v_{n-1}-v_0+(p-1)v_{n-1}\\
        &=1+p v_{n-1}.
\end{aligned}
\]
Hence every contribution in $p^{r+1}S_n$ has valuation strictly larger than $p v_{n-1}$.

If $1<n<e$, Proposition~\ref{prop:pure-decomp} gives
\[
        a_{p^n}(r)=-\gamma_{n,e}a_{p^{n-1}}(r)^p+pR_n+p^{r+1}S_n.
\]
The two remainder estimates above show that
\[
        a_{p^n}(r)=-\gamma_{n,e}a_{p^{n-1}}(r)^p+\bigo\bigl(p^{p v_{n-1}+1}\bigr),
\]
so
\[
        v_n=p v_{n-1}=p w_{n-1}=w_n.
\]
This gives the case $n<e$ in (P).

Now assume $n\ge e$.  Substituting the same remainder bounds into Proposition~\ref{prop:pure-decomp}, we obtain
\[
        a_{p^n}(r)=\alpha_{n,e}a_{p^{n-e}}(r)-\gamma_{n,e}a_{p^{n-1}}(r)^p+\bigo\bigl(p^{p v_{n-1}+1}\bigr).
\]
By the already-treated pure-power levels, we have $v_{n-e}=w_{n-e}$ and $v_{n-1}=w_{n-1}$.  Therefore the two main terms have valuations
\[
        A_{n,e}=\Delta_{n,e}+v_{n-e}=\Delta_{n,e}+w_{n-e},
        \qquad
        B_{n,e}=p v_{n-1}=p w_{n-1}.
\]
By Lemma~\ref{lem:abstract-branch-pe}, these two valuations are never equal.  The no-tie assertion is essential here: because the two candidate leading terms have distinct $p$-adic valuations, no cancellation can raise the valuation of $a_{p^n}(r,e)$.  Since both remainder terms have valuation strictly larger than $B_{n,e}$, the dominant term is exactly the smaller of the $A$- and $B$-branches.  Therefore
\[
        v_n=\min\{A_{n,e},B_{n,e}\}=w_n.
\]
If $A_{n,e}<B_{n,e}$, the leading term comes from the $A$-branch and
\[
        a_{p^n}(r)=\alpha_{n,e}a_{p^{n-e}}(r)+\bigo\bigl(p^{v_n+1}\bigr).
\]
If $B_{n,e}<A_{n,e}$, the leading term comes from the $B$-branch and
\[
        a_{p^n}(r)=-\gamma_{n,e}a_{p^{n-1}}(r)^p+\bigo\bigl(p^{v_n+1}\bigr).
\]
Thus the pure-power clause (P) is proved at level $p^n$, and in all cases $v_n\le p v_{n-1}$.

Once these inequalities hold up to level $n$, the carry argument gives (S) for all integers whose base-$p$ expansion uses only the digits $1,p,\ldots,p^n$: replacing $p$ copies of $p^j$ by one copy of $p^{j+1}$ can only decrease or preserve the total weight because $v_{j+1}\le p v_j$.

\smallskip
\noindent\emph{Divisible non-pure indices.}
Let $k=pM$ with $M\ge1$ and $k$ not a power of $p$.  Write
\[
        M=M_0+M_1p+\cdots+M_sp^s,
        \qquad
        0\le M_i\le p-1.
\]
Assume (L) and (D) hold for every $n'<k$, and (P) holds for every pure-power level below $k$.  Then (S) is available for all integers $\le k$: all pure-power levels $p^j\le k$ have already been treated, so the inequalities $v_{j+1}\le p v_j$ are known up to the relevant level, and the carry argument therefore gives subadditivity for every sum with total at most $k$.

We have the following shift inequality.  If $M'\ge1$ and $p^eM'\le k$, then
\begin{equation}\label{eq:shift-ineq-pe}
        \Lambda(p^eM')\le \mu_eM'-e+\Lambda(M'),
\end{equation}
and the inequality is strict if $M'$ is not a power of $p$.  Indeed, writing $M'=\sum_i M'_ip^i$ and using the already-proved pure-power recurrence gives
\[
\begin{aligned}
        \Lambda(p^eM')
        &=\sum_i M_i'v_{i+e}\\
        &\le \sum_i M_i'\bigl(\mu_ep^i-e+v_i\bigr)\\
        &=\mu_eM'-e\sum_i M_i'+\Lambda(M').
\end{aligned}
\]
Since $M'\ge1$, one has $\sum_i M_i'\ge1$, which proves \eqref{eq:shift-ineq-pe}.  If $M'$ is not a power of $p$, then $\sum_i M_i'\ge2$, so the inequality is strict.

If $q\nmid k$, then $A_k[x^k]=0$.  If $k=qM'$, then $M'$ is not a power of $p$, so \eqref{eq:shift-ineq-pe} is strict.  Hence
\[
        \vp\bigl(A_k[x^k]\bigr)\ge \mu_eM'-e+\Lambda(M')>\Lambda(k).
\]
For the $C$-term, every monomial coefficient is divisible by $p$ because $p\mid k$; see Proposition~\ref{prop:C-global}.  Therefore
\[
        \vp\bigl(p^rC_k[x^k]\bigr)\ge 1+r+\Lambda(k-1)>\Lambda(k)
\]
by (S), since $k=(k-1)+1$ and $\Lambda(1)=r$.  Thus
\begin{equation}\label{eq:D-first-reduction-pe}
        a_k(r)\equiv -B_k[x^k]\pmod{p^{\Lambda(k)+1}}.
\end{equation}

If a monomial in $B_k[x^k]$ contains some positive index not divisible by $p$, then, because $k\ne p$, Proposition~\ref{prop:B-unit-divisible} implies that its coefficient is divisible by $p$.  The induction hypothesis (L) and (S) give product valuation at least $\Lambda(k)$, and the coefficient contributes one extra factor of $p$.  Hence modulo $p^{\Lambda(k)+1}$ one may keep only the divisible truncation and write
\begin{equation}\label{eq:divisible-truncation-pe}
        B_k[x^k]\equiv
        \frac{k!}{p^e}[y^M]\left(\sum_{i=0}^{M-1}\frac{a_{pi}(r)}{(pi)!}y^i\right)^{p^e}
        \pmod{p^{\Lambda(k)+1}}.
\end{equation}
Here the index $i=0$ is harmless: both sides are $a_0(r)=m(0)=1$.  For $1\le i<M$, either $pi$ is a pure power, in which case $a_{pi}(r)=m(pi)$ by definition, or $pi$ is divisible and non-pure, in which case the induction hypothesis gives
\[
        a_{pi}(r)\equiv m(pi)\pmod{p^{\Lambda(pi)+1}}.
\]
Consider a monomial in the expansion of \eqref{eq:divisible-truncation-pe} in which at least one factor $a_{pi}(r)$ is replaced by its error term of valuation at least $\Lambda(pi)+1$.  Every remaining factor has valuation at least its digit weight $\Lambda(pj)$, and the scalar coefficient is $p$-integral by Lemma~\ref{lem:B-integrality}.  If the corresponding multiplicities are $\eta_i$, then subadditivity gives
\[
        1+\sum_i \eta_i\Lambda(pi)\ge 1+\Lambda\!\left(\sum_i \eta_i pi\right)=1+\Lambda(k),
\]
so the whole contribution is $\bigo\bigl(p^{\Lambda(k)+1}\bigr)$.  Therefore
\begin{equation}\label{eq:replace-by-mr-pe}
        B_k[x^k]\equiv
        \frac{k!}{p^e}[y^M]\left(\sum_{i=0}^{M-1}\frac{m(pi)}{(pi)!}y^i\right)^{p^e}
        \pmod{p^{\Lambda(k)+1}}.
\end{equation}
The already-proved pure-power part of the induction gives
\[
        v_{j+2}\le pv_{j+1}
        \qquad (0\le j\le s-1),
\]
so Proposition~\ref{prop:initial-form-B-pe} applies to the polynomial $\mathcal B_{M,r,e}(Y)$.  The right-hand side of \eqref{eq:replace-by-mr-pe} is obtained from $\mathcal B_{M,r,e}(Y)$ by the filtered specialization
\[
        Y_j\longmapsto a_{p^{j+1}}(r)
        \qquad (j\ge0),
\]
for which $\vp(a_{p^{j+1}}(r))=v_{j+1}$ by definition.  Hence a term of filtered degree $>\Lambda(k)$ specializes to $\bigo\bigl(p^{\Lambda(k)+1}\bigr)$, while the degree-$\Lambda(k)$ class $-Y^{[M]}$ specializes to $-m(k)$.  Therefore
\[
        B_k[x^k]\equiv -m(k)\pmod{p^{\Lambda(k)+1}}.
\]
Combining this with \eqref{eq:D-first-reduction-pe} gives
\[
        a_k(r)\equiv m(k)\pmod{p^{\Lambda(k)+1}},
\]
which proves (D) at the index $k$.

\smallskip
\noindent\emph{The global lower bound.}
Fix $k\ge1$.  Assume (L) holds for every $n'<k$, (D) holds for every divisible non-pure $n'<k$, and (P) holds for every pure-power level $p^j\le k$.  Then (S) is available for all integers $\le k$: all relevant pure-power levels have already been treated, so the inequalities $v_{j+1}\le p v_j$ are known up to the required height, and the carry argument gives subadditivity for every sum with total at most $k$.

If $k=p^n$ is a pure power, then (P) gives
\[
        \vp\bigl(a_{p^n}(r)\bigr)=v_n=\Lambda(p^n),
\]
so the lower bound holds with equality.  We may therefore assume that $k$ is not a pure power.  We use the recursion \eqref{eq:ABC-r}.

If $q\nmid k$, then $A_k[x^k]=0$.  If $k=qM$, then by the induction hypothesis and \eqref{eq:shift-ineq-pe},
\[
        \vp\bigl(A_k[x^k]\bigr)\ge \mu_eM-e+\Lambda(M)\ge \Lambda(qM)=\Lambda(k).
\]

Take a monomial $\gamma_B(\eta)a^\eta$ in $B_k[x^k]$.  By Lemma~\ref{lem:B-integrality}, $\gamma_B(\eta)\in\Zp$.  By the induction hypothesis,
\[
        \vp(a^\eta)\ge \sum_{n=0}^{k-1}\eta_n\Lambda(n).
\]
Since $\sum_{n=0}^{k-1}\eta_n n=k$, (S) gives
\[
        \sum_{n=0}^{k-1}\eta_n\Lambda(n)\ge \Lambda(k).
\]
Hence every $B$-monomial has valuation at least $\Lambda(k)$.

Take a monomial $\gamma_C(\eta)a^\eta$ in $C_k[x^k]$.  By Proposition~\ref{prop:C-global}, the coefficient $\gamma_C(\eta)$ is $p$-integral, and by the induction hypothesis,
\[
        \vp(a^\eta)\ge \sum_{n=0}^{k-1}\eta_n\Lambda(n)\ge \Lambda(k-1),
\]
because $\sum_{n=0}^{k-1}\eta_n n=k-1$.  Therefore
\[
        \vp\bigl(p^r\gamma_C(\eta)a^\eta\bigr)
        \ge r+\Lambda(k-1)\ge \Lambda(k),
\]
where the last inequality is the special case $k=(k-1)+1$ of (S).  Since every summand in \eqref{eq:ABC-r} has valuation at least $\Lambda(k)$, so does $a_k(r)$.  This gives (L) and finishes the induction proof of Theorem~\ref{thm:higher-structure}.
\end{proof}

\begin{proposition}[Pure-power branch pattern]\label{prop:branch-pattern-pe}
Let $s=s_{r,e}=\lceil r/e\rceil$.  The branch word for the pure-power recursion is
\[
        (B^{e-1}A)^sB^\infty.
\]
More explicitly, the $A$-branch occurs exactly at $n=e,2e,\ldots,se$; every other layer is $B$-dominated.  Moreover
\begin{equation}\label{eq:v-explicit-periodic}
        v_{je+t}=p^t\left(r+\frac{p^{j e}-1}{p-1}-e\,j\right)
        \qquad(0\le j\le s,\ 0\le t\le e-1),
\end{equation}
and $v_n=p^{n-e s}v_{e s}$ for $n\ge e s$.
\end{proposition}
\begin{proof}
Let $(w_n)_{n\ge0}$ be the abstract sequence of Lemma~\ref{lem:abstract-branch-pe}.  By the pure-power clause of Theorem~\ref{thm:higher-structure}, the sequence $v_n(r,e)$ satisfies the same initial conditions and the same lag-$e$ recursion as $w_n$.  Hence $v_n(r,e)=w_n$ for every $n\ge0$.  The explicit branch word and the formula \eqref{eq:v-explicit-periodic} are therefore exactly those of Lemma~\ref{lem:abstract-branch-pe}.
\end{proof}

\begin{proposition}[Normalized pure-power units]\label{prop:pure-units-pe}
For every $n\ge0$,
\[
        p^{-v_n(r,e)}a_{p^n}(r,e)\equiv \varepsilon_{r,e}(n)\pmod p,
        \qquad
        \varepsilon_{r,e}(n)=(-1)^{1+eN_{r,e}(n)}.
\]
\end{proposition}
\begin{proof}
The base case $n=0$ is \eqref{eq:a1-r}.  The case $n=1$ is Lemma~\ref{lem:first-two-pure-pe}.  Along a $B$-branch,
\[
        a_{p^n}(r)\equiv-\gamma_{n,e}a_{p^{n-1}}(r)^p
        \pmod{p^{v_n+1}},
\]
and $-\gamma_{n,e}\equiv1\pmod p$, so the normalized unit is unchanged.  Along an $A$-branch,
\[
        a_{p^n}(r)\equiv\alpha_{n,e}a_{p^{n-e}}(r)
        \pmod{p^{v_n+1}},
\]
and $p^{-\Delta_{n,e}}\alpha_{n,e}\equiv(-1)^e\pmod p$, so the normalized unit is multiplied by $(-1)^e$.  By Proposition~\ref{prop:branch-pattern-pe}, the number of $A$-branches up to level $n$ is $N_{r,e}(n)$, giving the formula.
\end{proof}

\begin{proof}[Proof of Theorem~\ref{thm:radius-pe}]
Part~\textup{(a)} is Proposition~\ref{prop:branch-pattern-pe}.  Part~\textup{(b)} is Proposition~\ref{prop:pure-units-pe}.  If $m=\sum_{i\ge0}m_ip^i$, then Theorem~\ref{thm:higher-structure}(D) gives the displayed leading term for $a_{pm}(r,e)$ when $m$ is not a power of $p$, and Proposition~\ref{prop:pure-units-pe} gives it when $m$ is a power of $p$.  Therefore $\vp(a_k(r,e))=\Lambda_{r,e}(k)$ for every divisible $k$, which is part~\textup{(c)}.

By Proposition~\ref{prop:branch-pattern-pe}, $v_n=\lambda_{r,e}p^n$ for all $n\ge es_{r,e}$.  Hence for $k=\sum k_ip^i$ the difference
\[
        \Lambda_{r,e}(k)-\lambda_{r,e}k
\]
depends only on the finitely many digits below level $es_{r,e}$ and is $\bigo(1)$.  This gives the valuation asymptotic in part~\textup{(c)}.

For the radius, write $f_{r,e}(x)=\sum_{N\ge1}b_Nx^N$ with $b_{k+1}=a_k(r,e)/k!$.  Since $N=k+1$, replacing $N$ by $k$ in the limsup defining the radius does not change the limit.  By the global lower bound and Legendre's formula,
\[
        \vp\left(\frac{a_k(r,e)}{k!}\right)
        \ge \Lambda_{r,e}(k)-\frac{k-\Sp(k)}{p-1}
        \ge -\theta_{r,e}k+\bigo(1).
\]
Thus $\rho(f_{r,e})\ge p^{-\theta_{r,e}}$.  Along $k=p^n$ with $n\ge es_{r,e}$,
\[
        \vp\left(\frac{a_{p^n}(r,e)}{(p^n)!}\right)
        =\lambda_{r,e}p^n-\frac{p^n-1}{p-1}
        =-\theta_{r,e}p^n+\frac1{p-1},
\]
so the opposite inequality holds.  Therefore $\rho(f_{r,e})=p^{-\theta_{r,e}}$.
\end{proof}

\section{Tail-stable extensions}\label{sec:tail-stability}

\begin{lemma}[Tail terms have controlled lower index and integral scalar coefficients]\label{lem:tail-structure}
Let $\vartheta_h\in\Qp$ for all $h\ge1$, and consider the formal germ
\[
        \widetilde\varphi_{r,e}(x)=x^q+qp^r x^{q+1}+q\sum_{h\ge1}\vartheta_h x^{q+1+h},
        \qquad q=p^e,
\]
with inverse B\"ottcher coordinate
\[
        \widetilde f_{r,e}(x)=x\sum_{k\ge0}\frac{\widetilde a_k(r,e)}{k!}x^k,
        \qquad
        \widetilde\varphi_{r,e}\bigl(\widetilde f_{r,e}(x)\bigr)=\widetilde f_{r,e}(x^q).
\]
In the coefficient recursion for $\widetilde a_k(r,e)$, the extra contribution coming from the tail term $q\vartheta_hx^{q+1+h}$ has the form $\vartheta_hT_{k,h}(\widetilde a)$, where $T_{k,h}(\widetilde a)=0$ for $k<h+1$, and every monomial of $T_{k,h}(\widetilde a)$ is
\[
        k!\binom{q+1+h}{\eta_0,\eta_1,\ldots}
        \prod_{n\ge0}\frac{\widetilde a_n(r,e)^{\eta_n}}{(n!)^{\eta_n}}
\]
with
\[
        \sum_{n\ge0}\eta_n=q+1+h,
        \qquad
        \sum_{n\ge0}n\eta_n=k-(h+1).
\]
In particular, every scalar coefficient of $T_{k,h}(\widetilde a)$ is an integer, hence $p$-integral, and every monomial of $T_{k,h}(\widetilde a)$ is a product of lower coefficients with total lower index $k-(h+1)$.
\end{lemma}
\begin{proof}
If $k<h+1$, then the coefficient of $x^k$ in
\[
        x^{q+1+h}\left(\sum_{n\ge0}\widetilde a_n(r,e)x^n/n!\right)^{q+1+h}
\]
vanishes.  Hence $T_{k,h}(\widetilde a)=0$.  Assume now that $k\ge h+1$, and set $m:=k-(h+1)$.  Then
\[
        T_{k,h}(\widetilde a)=k![x^m]\left(\sum_{n\ge0}\frac{\widetilde a_n(r,e)}{n!}x^n\right)^{q+1+h}.
\]
The multinomial expansion gives the displayed monomials, with
\[
        \sum_{n\ge0}\eta_n=q+1+h,
        \qquad
        \sum_{n\ge0}n\eta_n=m.
\]
Hence the total lower index is $m=k-(h+1)$.  Set $Q:=q+1+h$.  Then the scalar coefficient may be written as
\[
        k!\binom{Q}{\eta_0,\eta_1,\ldots}\prod_{n\ge0}\frac1{(n!)^{\eta_n}}
        =\frac{k!}{m!}\cdot \frac{Q!}{\eta_0!}\cdot
        \frac{m!}{\prod_{n\ge1}\eta_n!(n!)^{\eta_n}}.
\]
Here $k!/m!\in\mathbb Z$ and $Q!/\eta_0!\in\mathbb Z$.  The last factor counts set partitions of an $m$-element set into blocks of positive sizes $n\ge1$, repeated $\eta_n$ times, and is therefore an integer.  Hence every scalar coefficient of $T_{k,h}(\widetilde a)$ is an integer.
\end{proof}

\begin{proposition}[Small-tail stability in the higher fibers]\label{prop:small-tail-higher}
Let $r\ge1$ and let $\vartheta_h\in\Qp$ for all $h\ge1$.  Consider the formal germ
\[
        \widetilde\varphi_{r,e}(x)=x^q+qp^r x^{q+1}+q\sum_{h\ge1}\vartheta_h x^{q+1+h},
        \qquad q=p^e,
\]
where
\[
        \vp(\vartheta_h)\ge \Lambda_{r,e}(h+1)+1
        \qquad(h\ge1).
\]
Write
\[
        \widetilde f_{r,e}(x)=x\sum_{k\ge0}\frac{\widetilde a_k(r,e)}{k!}x^k,
        \qquad
        \widetilde\varphi_{r,e}\bigl(\widetilde f_{r,e}(x)\bigr)=\widetilde f_{r,e}(x^q),
\]
and define
\[
        \widetilde m_{r,e}(k):=\prod_{i\ge0}\widetilde a_{p^i}(r,e)^{k_i}
        \qquad\text{for }k=\sum_{i\ge0}k_ip^i.
\]
Since only finitely many $h$ contribute at each fixed coefficient degree, the discussion is formal coefficient-wise.  Then the conclusions of Theorem~\ref{thm:higher-structure} remain valid for the perturbed coefficients $\widetilde a_k(r,e)$, with $\widetilde m_{r,e}(k)$ in place of $m_{r,e}(k)$.  More precisely:
\begin{enumerate}[label=\textup{(\alph*)},leftmargin=2.2em]
\item $\vp(\widetilde a_k(r,e))\ge \Lambda_{r,e}(k)$ for every $k\ge1$.
\item If $p\mid k$ and $k$ is not a power of $p$, then
\[
        \widetilde a_k(r,e)\equiv \widetilde m_{r,e}(k)\pmod{p^{\Lambda_{r,e}(k)+1}}.
\]
\item For every $n\ge0$,
\[
        \vp\bigl(\widetilde a_{p^n}(r,e)\bigr)=v_n(r,e),
        \qquad
        p^{-v_n(r,e)}\widetilde a_{p^n}(r,e)\equiv \varepsilon_{r,e}(n)\pmod p.
\]
Consequently, the pure-power branch word, the valuation asymptotic, and the radius formula of Propositions~\ref{prop:branch-pattern-pe} and~\ref{prop:pure-units-pe}, and Theorem~\ref{thm:radius-pe} remain unchanged.
\end{enumerate}
\end{proposition}
\begin{proof}
By \cite[Theorem~3(b)]{SalernoSilverman}, the perturbed coefficients $\widetilde a_k(r,e)$ are $p$-integral.  Comparing coefficients in the perturbed B\"ottcher equation gives
\begin{equation}\label{eq:tail-recursion-higher}
        \widetilde a_k(r,e)=A_k[\widetilde a]-B_k[\widetilde a]-p^rC_k[\widetilde a]-\sum_{1\le h<k}\vartheta_hT_{k,h}(\widetilde a),
\end{equation}
where $A_k$, $B_k$, and $C_k$ are the same $A$-, $B$-, and $C$-terms as in the clean family.

First we check that the low-index input is unchanged.  Write
\[
\begin{aligned}
        \widetilde g_r(x)&:=\sum_{j\ge0}\frac{\widetilde a_j(r,e)}{j!}x^j,\qquad
        A(x):=p^r x\widetilde g_r(x),\\
        B(x)&:=\sum_{h\ge1}\vartheta_hx^{h+1}\widetilde g_r(x)^{h+1}.
\end{aligned}
\]
Then
\[
        \widetilde g_r(x^q)=\widetilde g_r(x)^q\bigl(1+qA(x)+qB(x)\bigr),
\]
and therefore
\[
        \log \widetilde g_r(x^q)=q\log \widetilde g_r(x)+\log\bigl(1+qA(x)+qB(x)\bigr).
\]
Fix $1\le n<p$ and argue by induction on $n$, exactly as in Lemma~\ref{lem:low-index-pe}.  After dividing the coefficient relation by $q$, the clean linear term is still $-p^r[x^{n-1}]\widetilde g_r(x)$.  It remains to check that every contribution containing at least one tail factor is $O(p^{nr+1})$.

For the linear tail term, a summand with $h+1\le n$ contributes
\[
        \vartheta_h[x^{n-(h+1)}]\widetilde g_r(x)^{h+1}.
\]
Because $h+1<p$, one has $\Lambda_{r,e}(h+1)=(h+1)r$, so
\[
        \vp(\vartheta_h)\ge \Lambda_{r,e}(h+1)+1=(h+1)r+1.
\]
A monomial contributing to $[x^{n-(h+1)}]\widetilde g_r(x)^{h+1}$ is a product of lower coefficients with total lower index $n-(h+1)$, hence by the low-index induction has valuation at least $r(n-(h+1))$.  Thus every linear tail contribution has valuation at least
\[
        (h+1)r+1+r(n-(h+1))=nr+1.
\]

For the nonlinear terms, after dividing by $q$ we obtain
\[
        \sum_{m\ge2}(-1)^{m-1}\frac{q^{m-1}}{m}(A(x)+B(x))^m.
\]
The terms involving only $A(x)$ are exactly the clean nonlinear terms already handled in Lemma~\ref{lem:low-index-pe}.  Consider instead a term of order $m\ge2$ containing at least one factor from $B(x)$ and contributing to the coefficient of $x^n$.  By the low-index induction, every coefficient of $x^s$ in $A(x)=p^r x\widetilde g_r(x)$ has valuation at least $rs$.  Likewise, for each fixed $h\ge1$, every coefficient of $x^s$ in $\vartheta_hx^{h+1}\widetilde g_r(x)^{h+1}$ has valuation at least $rs+1$: indeed, the tail factor contributes at least $(h+1)r+1$, and the remaining coefficient factors contribute at least $r(s-(h+1))$.  Hence any product of total degree $n$ containing at least one $B$-factor has valuation at least $nr+1$.  Since $m\le n<p$, the denominator $m$ is a $p$-adic unit, so the prefactor $q^{m-1}/m$ has nonnegative $p$-adic valuation.  Therefore every nonlinear tail contribution is also $O(p^{nr+1})$.

Thus the proof of Lemma~\ref{lem:low-index-pe} is unchanged, so $\vp(\widetilde a_n(r,e))\ge nr$ for $1\le n<p$.  At $k=p$, the same estimate and Lemma~\ref{lem:tail-structure} show that every tail term is $O(p^{pr+1})$, so the proof of Lemma~\ref{lem:first-two-pure-pe} gives
\[
        \vp\bigl(\widetilde a_p(r,e)\bigr)=pr,
        \qquad
        p^{-pr}\widetilde a_p(r,e)\equiv-1\pmod p.
\]

Now assume that $\vp(\widetilde a_n(r,e))\ge \Lambda_{r,e}(n)$ is already known for every $n<k$.  By Lemma~\ref{lem:tail-structure}, every monomial of $T_{k,h}(\widetilde a)$ is a product of lower coefficients with total lower index $k-(h+1)$ and has $p$-integral scalar coefficient.  Repeatedly applying the clean-family subadditivity from Theorem~\ref{thm:higher-structure}\textup{(4)} therefore gives
\[
        \vp\bigl(T_{k,h}(\widetilde a)\bigr)\ge \Lambda_{r,e}(k-(h+1)).
\]
Hence
\[
        \vp\bigl(\vartheta_hT_{k,h}(\widetilde a)\bigr)
        \ge \Lambda_{r,e}(h+1)+1+\Lambda_{r,e}(k-(h+1))
        \ge \Lambda_{r,e}(k)+1.
\]
Thus the whole tail sum in \eqref{eq:tail-recursion-higher} is $O(p^{\Lambda_{r,e}(k)+1})$; when $k=p^n$, this is $O(p^{v_n(r,e)+1})$.

We now run the same induction package as in the proof of Theorem~\ref{thm:higher-structure}.  In the pure-power step, the tail sum is absorbed into the existing $O(p^{v_n+1})$ error term from Proposition~\ref{prop:pure-decomp}, so the same no-tie branch comparison applies.  In the divisible non-pure step, the tail sum is absorbed into $O(p^{\Lambda_{r,e}(k)+1})$, so the same initial-form $B$-term argument gives the same leading monomial $\widetilde m_{r,e}(k)$.  In the global lower-bound step, the tail sum lies strictly above the required lower bound, so the same $A$--$B$--$C$ estimate gives the desired valuation bound.  Consequently the analogues of Theorem~\ref{thm:higher-structure}\textup{(1)}--\textup{(4)} hold for $\widetilde a_k(r,e)$, with $\widetilde m_{r,e}(k)$ in the divisible non-pure clause.  The final assertions about the pure-power branch word, normalized units, valuation asymptotic, and radius then follow exactly as in Propositions~\ref{prop:branch-pattern-pe} and~\ref{prop:pure-units-pe}, and Theorem~\ref{thm:radius-pe}.
\end{proof}

\begin{proposition}[Small $p$-divisible tails in the special fiber]\label{prop:small-tail-special}
Let $\vartheta_h\in\Qp$ for all $h\ge1$, and consider the formal germ
\[
        \widetilde\varphi_{0,e}(x)=x^q+qx^{q+1}+q\sum_{h\ge1}\vartheta_h x^{q+1+h},
        \qquad q=p^e,
\]
with $\vartheta_h\in p\Zp$ for all $h\ge1$.  Write
\[
        \widetilde f_{0,e}(x)=x\sum_{k\ge0}\frac{\widetilde a_k(e)}{k!}x^k,
        \qquad
        \widetilde\varphi_{0,e}\bigl(\widetilde f_{0,e}(x)\bigr)=\widetilde f_{0,e}(x^q),
\]
and let $a_k=a_k(0,e)$ denote the clean special-fiber coefficients.  Then
\[
        \widetilde a_k(e)\equiv a_k\pmod p
        \qquad(k\ge0).
\]
In particular, the digit-sum law and the residue-class congruences of Theorems~\ref{thm:digit-sum} and~\ref{thm:SS-special} remain valid for $\widetilde a_k(e)$.
\end{proposition}
\begin{proof}
Since $\widetilde\varphi_{0,e}(x)\in x^q+qx^{q+1}\Zp[\![x]\!]$, \cite[Theorem~3(b)]{SalernoSilverman} gives $\widetilde a_k(e)\in\Zp$ for every $k$.  Comparing coefficients in the perturbed B\"ottcher equation gives
\begin{equation}\label{eq:tail-recursion-special}
        \widetilde a_k(e)=A_k[\widetilde a]-B_k[\widetilde a]-C_k[\widetilde a]-\sum_{1\le h<k}\vartheta_hT_{k,h}(\widetilde a),
\end{equation}
where $A_k$, $B_k$, and $C_k$ are the special-fiber terms of Sections~\ref{sec:recursion}--\ref{sec:proof-main-special}.  We claim by induction on $k$ that $\widetilde a_k(e)\equiv a_k\pmod p$.  The case $k=0$ is tautological.

Assume the claim known for all indices $<k$.  The $A$-coefficients are $p$-integral, all $B$-coefficients are $p$-integral by Lemma~\ref{lem:B-integrality}, and all $C$-coefficients are $p$-integral by Proposition~\ref{prop:C-global}.  Hence the induction hypothesis gives
\[
        A_k[\widetilde a]\equiv A_k[a],
        \qquad
        B_k[\widetilde a]\equiv B_k[a],
        \qquad
        C_k[\widetilde a]\equiv C_k[a]
        \pmod p.
\]
On the other hand, $\vartheta_h\in p\Zp$ and Lemma~\ref{lem:tail-structure} shows that every scalar coefficient of $T_{k,h}(\widetilde a)$ is $p$-integral, so
\[
        \vartheta_hT_{k,h}(\widetilde a)\equiv0\pmod p
        \qquad(1\le h<k).
\]
Reducing \eqref{eq:tail-recursion-special} modulo $p$ therefore gives
\[
        \widetilde a_k(e)\equiv A_k[a]-B_k[a]-C_k[a]=a_k\pmod p.
\]
This completes the induction.  The final sentence follows immediately from Theorems~\ref{thm:digit-sum} and~\ref{thm:SS-special}.
\end{proof}

\begin{proof}[Proof of Theorem~\ref{thm:tail-main}]
Part~\textup{(a)} is Proposition~\ref{prop:small-tail-special}, and part~\textup{(b)} is Proposition~\ref{prop:small-tail-higher}.
\end{proof}

\end{document}